\documentclass{amsart}
\usepackage{amsmath}
\usepackage{amsfonts}
\usepackage{amssymb}
\usepackage{amsthm}
\usepackage{mathtools}
\usepackage{url}
\usepackage{tikz-cd}
\usepackage{arydshln}
\usepackage{enumitem}
\usepackage{dsfont}
\usepackage{comment}
\usepackage{mathrsfs}

\newcommand{\NN}{\mathbb{N}}
\newcommand{\ZZ}{\mathbb{Z}}

\newcommand{\FF}{\mathbb{F}}
\newcommand{\U}{\mathcal{U}}
\newcommand{\V}{\mathcal{V}}
\newcommand{\W}{\mathcal{W}}
\newcommand{\Z}{\mathcal{Z}}

\theoremstyle{theorem}
\newtheorem{definition}{Definition}[section]
\newtheorem{theorem}{Theorem}[section]
\newtheorem{corollary}{Corollary}[section]
\newtheorem{conjecture}{Conjecture}[section]
\newtheorem{lemma}{Lemma}[section]

\theoremstyle{definition}
\newtheorem{case}{Case}
\newtheorem{example}{Example}
\newtheorem{observation}{Observation}[section]

\begin{document}
\title{Structures in Additive Sequences}

\author{Borys Kuca}
\email[Borys Kuca]{borys.kuca@yale.edu}

\address{Morse College, Yale University, 304 York St, New Haven, CT 06511}

\begin{abstract} Consider the sequence $\mathcal{V}(2,n)$ constructed in a greedy fashion by setting $a_1 = 2$, $a_2 = n$ and defining $a_{m+1}$
as the smallest integer larger than $a_m$ that can be written as the sum of two (not necessarily distinct) earlier terms in exactly one way; the sequence $\mathcal{V}(2,3)$, for example, is given by 
$$ \mathcal{V}(2,3) = 2,3,4,5,9,10,11,16,22,\dots$$
We prove that if $n \geqslant 5$ is odd, then the sequence $\mathcal{V}(2,n)$ has exactly two even terms $\left\{2,2n\right\}$ if and only if $n-1$ is not a power of 2. We also show that in this case, $\mathcal{V}(2,n)$ eventually becomes a union of arithmetic progressions. If $n-1$ is a power of 2, then there is at
least one more even term $2n^2 + 2$ and we conjecture there are no more even terms. In the proof, we display an interesting connection between $\V(2,n)$ and Sierpinski Triangle. We prove several other results, discuss a series of striking phenomena and pose many problems. This relates to existing results of Finch, Schmerl \& Spiegel and a classical family of sequences defined by Ulam.
\end{abstract}


\maketitle

\section{Introduction}
\subsection{Introduction} This paper is concerned with curious structures emerging in integer sequences defined by simple additive relations. 
We define the integer sequence $\mathcal{V}(a,b)$ by adding, in a greedy fashion, the least positive integer greater than previous terms of the sequence that can be written as the sum of two earlier (not necessarily distinct) elements of the sequence in a unique way. Examples are given by the sequences
\begin{align*}
\V(1,2) &=\{1,2,3,5,7,9,11,13,15,17,19,21,23,25,27,...\}\\
\V(2,3) &=\{2,3,4,5,9,10,11,16,22,24,28,29,30,37,42,...\}
\end{align*}

It turns out that $\V$-sequences have, perhaps surprisingly, a wealth of intriguing structures some of which we will describe in this paper. 

\begin{theorem}[Main result]\label{(2,n) theorem 1}
Let $n\geqslant 5$ be odd. If $n-1$ is not a power of 2, then $\V(2,n)$ contains exactly two even terms, $2$ and $2n$, and eventually becomes the union of finitely many arithmetic progressions. If
$n-1$ is a power of 2, then it contains at least the three even numbers $\left\{2,2n,2n^2+2\right\}$.
\end{theorem}

We conjecture that in the second case, $n-1$ being a power of 2, the sequence contains exactly three even elements. If that were true, then we can show that it can  also be
written as a union of finitely many arithmetic progressions except for finitely many initial terms. This statement is contrasted by the following conjecture.

\begin{conjecture}\label{Conjecture about Regular Modified Ulam Sequences}
Let $a,b$ be relatively prime. Then $\V(a,b)$ is the union of finitely many arithmetic progressions if and only if  $(a,b) = (1,2)$, $(a,b) = (1,3)$ or $a$ is even and $b > 2a$ is odd.
\end{conjecture}

We also discovered that a result of Finch, proven in a different context and for a different type of sequence, holds in our case as well.
\begin{theorem}[Finch]\label{Finch's criterion}
A $\V$-sequence with only finitely many even terms eventually becomes a union of finitely many arithmetic progressions.
\end{theorem}
We also obtain a series of other results as a byproduct. One statement is obtained by using Freiman homomorphisms to link the sequence to a geometric structure in $\mathbb{Z}^2_{\geq 0}$
and is as follows.
\begin{theorem}
If $a,b$ are relatively prime positive integers such that $a$ is even and $b>2a$, then $\V(a,b)$ has at least $2+\lfloor a/4\rfloor$ even terms.
\end{theorem}
There is a natural extension of this type of sequence by enforcing more complicated additive relationships. One could, for example, study greedy sequences where one adds the
smallest integer that can be uniquely written as $2x + y$ for a unique choice of distinct $x,y$ already in the sequence. We call these sequences $(2,1)$-sequences and will denote
the $(2,1)$-sequence with initial terms $a,b$ by $\Z_{(2,1)}(a,b)$. Of course, there seems to be nothing special about $(2,1)$ and one could study other constellations. 

\begin{theorem}
 The sequence  $\Z_{(2,1)}(1,3)$ is given by
$$\Z_{(2,1)}(1,3)=\{3,15\}\cup(4\ZZ_{\geqslant 0}+1)\setminus{\{9\}}=\{1, 3, 5, 13, 15, 17, 21, 25,...\}$$
\end{theorem}

This suggests that $\Z_{(2,1)}$ sequences can be simple and should be contrasted with the following purely empirical observation: the sequences $\Z_{(2,1)}(1,9)$ and $\Z_{(2,1)}(3,7)$ seem to have a positive upper and lower density that are different from each other. More precisely,  $\Z_{(2,1)}(1,9)$ seems to have upper density $\sim 0.123$ and lower density $\sim 0.107$. What is possibly even more surprising is that the density seems to fluctuate in a rather regular manner (see Figure 1). The sequence $\Z_{(2,1)}(3,7)$ seems to have upper density $\sim 0.122$ and lower density $\sim 0.106$.

\begin{figure}[h!]
\includegraphics[height = 0.25\textheight]{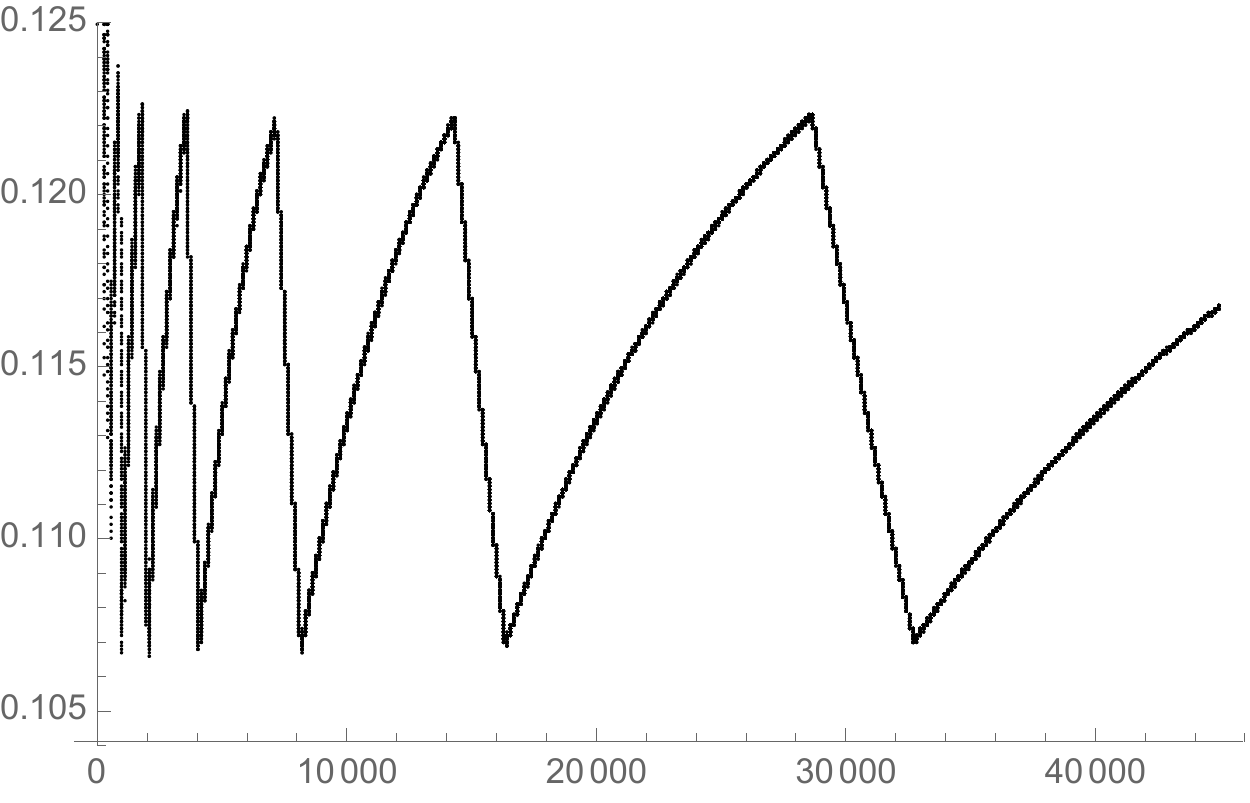}
\caption{The density of $\Z_{(2,1)}(1,9)$ on the first 45000 integers.}
\label{density19}
\end{figure}

We are not aware of any of these results being known. They do seem to indicate some rather interesting structure.

\subsection{Background.}\label{Background}  Our motivation for studying $\V-$sequences derives from recently renewed interest in a class of integer sequences defined by Stanislaw Ulam in 1964 (for reasons that are not entirely clear). An Ulam sequence $\U(a,b)$ starts with the elements $a, b$ and is then constructed by repeatedly adding the smallest integer that can be written as the sum of two \textit{distinct} earlier terms in a unique way. $\V-$sequences are defined by dropping the condition of having the earlier terms be distinct (so, in a certain sense, it should be simpler). Ulam himself only mentioned the sequence 
\begin{align*}
\U(1,2) &=\{1,2,3,4,6,8,11,13,16,18,26,...\}
\end{align*}
Ulam himself remarks that the sequence seems to be erratic but it is not entirely clear why he defined the sequence in the first place \cite{ulam_1964}. It was soon understood that different initial values can give rise to more structured sequences: some Ulam sequences have only finitely many even terms. Finch conjectured a characterization of initial conditions leading to sequences with only finitely many even terms and proved that Ulam sequences with this property become a finite union of arithmetic progressions after a finite transient phase \cite{finch_1992_1}. Some sequences on Finch's list have been shown to be regular; the regularity of others is still subject to conjecture \cite{schmerl_spiegel_1994, cassaigne_finch_1995, hinman_kuca_schlesinger_sheydvasser_2017}. There is recent renewed interest due to a curious empirical discovery of Steinerberger \cite{steinerberger_2016}: there seems to exist a real number $\lambda \sim 2.4\dots$ such that the elements of $\U(1,2)$ have strange clustering behavior in the sense of $a_n \mod\lambda$ having a non-uniform limit distribution that is compactly supported. More precisely, as was observed by Gibbs \cite{gibbs_2015}, the residues are concentrated in the middle third of the interval. Techniques developed in the study of $\V$-sequences allows us to prove two new results for Ulam sequences.

\begin{theorem}
Let $\U(a,b)$ be an Ulam sequence with $a,b$ relatively prime, and one of $a,b$ even. Then:
\begin{itemize}
\item if $a$ is even, then $\U(a,b)$ has at least $1+a/2$ even terms.
\item if $b$ is even, then $\U(a,b)$ has at least $1+b/2$ even terms.
\end{itemize}
\end{theorem}

\begin{theorem}
Let $\U(a,b)$ be an Ulam sequence with $a,b$ relatively prime. Then: 
\begin{itemize}
\item for any $c>(b+1)a$, $\{c, c+a, ..., c+ba\}\subset \U\implies c+(b+1)a\notin \U$
\item for any $c>b(a+1)$, $\{c, c+b, ..., c+ab\}\subset \U\implies c+(a+1)b\notin \U$
\end{itemize}
\end{theorem}

\subsection*{Notation}
 $\U$ always refers to Ulam sequences, and $\U(a,b)$ will denote the Ulam sequence generated by $a$ and $b$. $\V$ refers to $\V$-sequences while $\V(a,b)$ will denote the $\V$-sequence generated by $a$ and $b$. Finally, $\Z_{(a_1,...,a_n)}(b_1,...,b_n)$ denotes $(a_1,...,a_n)$-set generated by $(b_1,...,b_n)$. Note that in the literature, $\Z_{(1,...,1)}$, where $(1,...,1)$ consists of $n$ 1's, are also known as $(1,n)$-additive \cite{queneau_1972}\cite{finch_1992_1}. In particular, $\U=\Z_{(1,1)}$ is also called a $(1,2)$-additive sequence whereas $\Z_{(1,1,1)}$ is called a $(1,3)$-additive sequence.

\section{Basic Properties of $\V$-sequences}
We start with a basic observation.
\begin{theorem}\label{First two sequences} 
\begin{align*}
& 1) \;\V(1,2)=(2\ZZ_+-1)\cup\{2\}=\{1,2,3,5,7,...\} \\
& 2) \;\V(1,3)=(2\ZZ_+-1)\cup\{4\}=\{1,3,4,5,7,...\}
\end{align*}
\end{theorem}
\begin{proof}
For $\V(1,2)$, each odd number $2a+1\geqslant 3$ has a unique representation $2a+1=(2a-1)+2$. $4=2+2=3+1$, and each even number $2a\geqslant 6$ has at least two representations $2a=(2a-1)+1=(2a-3)+3$. By induction, we get the desired form of the sequence. For $\V(1,3)$, we start the sequence with $1,3, 4=1+3$. Then each odd number $2a+1\geqslant 5$ has a unique representation $2a+1=(2a-3)+4$, and each even number $2a\geqslant 6$ has at least two representations $2a=(2a-1)+1=(2a-3)+3$. By induction, we get the desired form of the sequence.
\end{proof}

Seeing that the first two $\V$-sequences both eventually become arithmetic progressions of consecutive odd numbers, one might be tempted to ask if other $\V$-sequences having two generators exhibit the same behavior. The answer is no.

\begin{theorem}\label{Regularity of modified Ulam set}
Let $\V=\V(a,b)$ be a $\V$-sequence, and suppose there exists $2c+1\in\ZZ_+$ s.t. $2c+1,2c+3,2c+5,...\in\V$. Then $\V=\V(1,2)$ or $\V=\V(1,3)$. 
\end{theorem}

First note that we cannot have a $\V$-sequence that would contain all elements of the form $2c,2c+2,2c+4,...$ for some $c\in\ZZ_+$ because then $4c+4=(2c+2)+(2c+2)=(2c+4)+2c$ would not be in the set because of multiple representations.
\begin{proof}
Suppose $2c+1$ is a minimal number s.t. $2c+1,2c+3,2c+5,...\in\V$. First note that $\V$ must contain precisely one even number. It must contain at least one even number because numbers of the form $2c+2n+1$ are all odd, hence they need an even summand. If however there are two even numbers $2k, 2l\in\V$, $k<l$, then $2c+2l+1=(2c+1)+2l=(2(c+l-k)+1)+2k$ has two representations. Let $2k$ be the unique even number in $\V$. Then all numbers of the form $2c+2n+1$ are either generators or have the representation $2c+2n+1=(2(c+n-k)+1)+2k$. $2(c+k-1)+1$ must be a generator, otherwise its representation would be $2(c+k-1)+1=(2(c-1)+1)+2k$, contradicting the minimality of $2c+1$. Thus all of $2c+1,2(c+1)+1,...,2(c+k-1)+1$ are generators. Assuming as we did that $\V$ has only two generators, this requires that $k=1$ or $2$. If $k=1$, then $2$ and $2c+1$ are the generators. Unless $2c+1=1$, $4c+2=(2c+1)+(2c+1)$ would be an even term distinct from $2$; hence the case $k=1$ implies that $\V=\V(1,2)$. If $k=2$, i.e. $2k=4$, then $2c+1$ and $2c+3$ have to be the generators. $4c+2=(2c+1)+(2c+1)$ will be a second even term unless $4c+2\leqslant4$, in which case $2c+1=1$, $2c+3=3$. This gives the case $\V(1,3)$.

%In particular, $2a+1=(2(a-k)+1)+2k$, and similarly for $2(a+1)+1,...,2(a+k-1)+1$. We will try to show that $2a+1=1$.
%Suppose not. Then none of the elements $1,3,...,2a-1$ is in $\V$ by minimality of $2a+1$. Hence $2a+1$, ..., $2(a+k-1)+1$ are all among the generators, together with $2k$. So $\V=\V(2k,2a+1,2a+3,...,2(a+k-1))$. Then $4k=2k+2k\in\V$ unless $4k\geqslant (2a+1)+(2a+1)=4a+2$. We must thus have $k>a$. If $k\neq 4a+2$, then $4a+2=(2a+1)+(2a+1)\in\V$ because it has precisely one representation. So we require $k=4a+2$. Then $4a+4=(2a+1)+(2a+3)\in\V$ because it also has one representation. But then we have two even elements in $\V$, $4a+4$ and $4a+2$, so we obtain a contradiction.
%Hence $2a+1=1$, then $1,3,5,...\in\V$. If $2\in\V$, then we obtain $\V=\V(1,2)$. If $2\notin\V$, then 3 is one of the generating elements of the sequence. Since 1 and 3 are among the generators, we have $4=1+3\in\V$ and then $\V=\V(1,3)$, i.e. $\V$ has no other nontrivial generators. 
\end{proof}
If we do not insist that $\V$ has two generators, we can find other sequences that eventually become an infinite arithmetic progression of period 2. For instance, $\V(1,2,3,9,11)$ and $\V(3,4,5,7)$ have this property. This phenomenon is however rare - it is more common for $\V$-sequences to eventually become a union of arithmetic progressions. We will call such sequences regular.

\begin{definition}\label{Definition of Regular modified Ulam sequences}
We say that an increasing integer sequence $(a_n)_{n=1}^\infty$ is regular if $(a_n)_{n=n_0}^\infty$ can be written as the finite union of arithmetic progressions for some $n_0\in\NN$. Equivalently, $(a_{n+1}-a_n)_{n=n_0}^\infty$ is periodic.
\end{definition}

Finch proved a sufficient condition for an Ulam sequence to be periodic \cite{finch_1992_2, finch_1992_1}. The same criterion holds for $\V$-sequences and, interestingly, Finch's proof
never actively uses the requirement that sums be distinct and carries over verbatim.

\begin{theorem}[Finch]\label{Finch's criterion}
A $\V$-sequence with only finitely many even terms is regular.
\end{theorem}
 Finch also conjectured which Ulam sequences have finitely many even terms.
Finch's theorem that finitely many even terms implies regularity is a statement in one direction. We conjecture that under suitable circumstances, the other direction is also true. We do not have an intuition for why the converse would be true, but there is no known counterexample to the contrary. 

\begin{conjecture}\label{Finch - second direction}
Let $a<b$ be relatively prime positive integers. Then the Ulam sequence $\U(a,b)$ (respectively, the $\V$-sequence $\V(a,b)$) is regular if and only if $\U(a,b)$ (respectively, $\V(a,b)$) has finitely many even terms. 
\end{conjecture}

\section{The Regularity of $\V(2,n)$}

We now prove our main result about the regularity of $\V(2,n)$. We recall the main statement as well as the main conjecture for the convenience of the reader.

\begin{theorem}\label{(2,n) theorem}
Let $\V:=\V(2,n)$, and $n\geqslant 5$ odd.
\begin{enumerate}
\item If $n-1$ is a power of 2, then $\V$ has at least three even terms: $2,2n,2n^2+2$.
\item Otherwise $\V$ has exactly two even terms: $2$ and $2n$. 
\end{enumerate}
\end{theorem}

\begin{conjecture}\label{(2,n) conjecture}
If $n-1$ is a power of 2, then $\V$ has precisely three even terms: $2,2n,2n^2+2$.
\end{conjecture}

Using Theorem \ref{Finch's criterion}, we obtain the following important statement about the regularity of $\V$:
\begin{corollary}\label{(2,n) corollary}
Let $\V:=\V(2,n)$, and $n\geqslant 5$ odd.
\begin{enumerate}
\item If $n-1$ is not a power of 2, then the sequence of differences $(u_{n+1} - u_n)$ eventually becomes periodic.
\item  If $n-1$ is a power of 2, and if Conjecture \ref{(2,n) conjecture} is true, then $\V$ is regular as well.
\end{enumerate}
\end{corollary}

 We begin by mimicking Schmerl and Spiegel's proof that the Ulam sequence $\U(2,n)$ has two even terms 2 and $2n+2$ for odd $n\geqslant 5$ \cite{schmerl_spiegel_1994}. Then we show where Schmerl and Spiegel's proof breaks when applied to the $\V$-sequence $\V(2,n)$. 

\begin{proof}[Proof of Theorem \ref{(2,n) theorem}]
The strategy behind the proof is that we assume that $\V$ has a third even term $x$, and then we find the necessary and sufficient conditions for this to be the case. We first determine the initial terms of the sequence, including two even terms 2 and $2n$ (Lemma \ref{first terms of (2,n)}). Then we show how the knowledge of initial terms sheds light on the structure of $\V\cap[x-(2n-1)n-2,x]$ (Lemmas \ref{(2,n) representations}-\ref{Right before x}). We subsequently show that the structure of $\V\cap[x-(2n-1)n-2,x]$ is related to Sierpinski's Triangle, thus having fractal behavior (Lemmas \ref{binary sequences, basic}-\ref{P_k vs. Q_k}). Using this fractal structure, we conclude that if $n-1$ is a power of 2, then $\V$ has a third even term $x=2n^2+2$ (Lemma \ref{third even term}). Finally, we show that for other odd values of $n$, $x$ has a second representation implying that $\V$ has no third even term (Lemmas \ref{S_l}-\ref{P_{2^k n}}). 

\begin{lemma}\label{first terms of (2,n)}
$\V\cap[2,5n+2]=\{2,n,n+2,n+4,...,2n-3,2n-1,2n,2n+1,2n+3,...,3n-4,3n-2,3n+2,3n+6,...,5n-8, 5n-4,5n+2\}$
\end{lemma}

Suppose $\V$ has more even terms than just $2$ and $2n$. Let $x$ be the least even positive integer in $\V$ greater than $2n$. We will show that either $x$ does not exist, or $x=2n^2+2$ precisely when $n-1$ is a power of 2.

\begin{lemma}\label{(2,n) representations}
If $2n<u<x$ and $u$ is odd, then $u\in\V\iff$ precisely one of $u-2$, $u-2n$ is in $\V$. 

Equivalently, $1(u)=1(u-2)+1(u-2n)$ where $1=1_\V$ is the indicator function of $\V$ for $2n<u<x$ and $u$ odd. 
\end{lemma}

This lemma is the main technical observation used in the proof. It is used so abundantly throughout the proof that we do not quote it. 
\begin{proof}
Since there are only two even summands that $u$ could have, $2$ and $2n$, $u$ is in $\V$ iff precisely one of $u-2$ and $u-2n$ is in $\V$.
\end{proof}
\begin{lemma}\label{(2,n) lemma 3}
If $r$ is an odd number s.t. $1\leqslant r\leqslant x-2n+2$, then there exists $0\leqslant i\leqslant n-1$ s.t. $r+2i\in \V$.
\end{lemma}
\begin{proof}
Suppose not, and let $r$ be the smallest odd positive integer st. $r+2i\notin\V$ for all $0\leqslant i\leqslant n-1$. Clearly $r\geqslant 3$. By assumption, $r-2+2i\in \V$ for some $0\leqslant i\leqslant n-1$. Since $r, r+2,...,r+2n-2\notin\V$, we must have $r-2$ in $\V$. Then precisely one of $r+2n-2, r+2n-4$ is in $\V$ which contradicts the assumption.
\end{proof}

Using Lemmas \ref{first terms of (2,n)}-\ref{(2,n) lemma 3}, we obtain valuable information about the structure of $\V\cap[x-(2n-1)n,x-n]$. This knowledge will be crucial for the rest of the proof of Theorem \ref{(2,n) theorem}.

\begin{lemma}\label{Right before x}
Suppose $\V$ has the third even term $x$. Then $x-n, x-3n$ are in $\V$ but none of $x-n-2, x-n-4, ..., x-3n+2$ is. 
\end{lemma}

\begin{proof}
Following the previous lemma, pick $r = x-3n+2$. By Lemma \ref{(2,n) lemma 3}, there exists $0\leqslant i\leqslant n-1$ s.t. $r+2i=x-3n+2+2i\in\V$. Then $x$ has the following Ulam representation: $x=(x-3n+2+2i)+(3n-2-2i)$. We will show under which conditions it must have a second representation. By Lemma \ref{(2,n) lemma 3}, at least one of $x-n,x-n-2,...,x-3n+2$ is in $\V$ - but if two of them are in $\V$, then $x$ will have two representations. Using the fact that the only even summands are $2$ and $2n$, we can also determine whether or not each of $x-3n,x-3n-2,...,x-5n+2$ is in $\V$ because $x-5n+2+2j$ is in $\V$ iff precisely one of $x-3n+2+2j, x-3n+2j$ is in $\V$. We have two cases:
\begin{itemize}
\item if $0\leqslant i < n-1$, $0\leqslant j\leqslant n-1$, then $x-5n+2+2j\in\V$ iff $j=i$ or $j=i-1$.
\item if $i=n-1$, $0\leqslant j\leqslant n-1$, then $x-5n+2+2j\in\V$ iff $j=0$ or $j=n-1$. 
\end{itemize}
\begin{case}
We have $x-3n+2+2i\in\V$ and either $x-5n+2+2i\in\V$ or $x-5n+2i\in\V$. One of $5n-2-2i, 5n-2i$ is in $\V$, hence $x$ has a second representation. Thus this case leads to a contradiction.
\end{case}
\begin{case}
We have $x-n\in \V$ and precisely one of $x-n-2,x-3n$ is in $\V$. Note that none of $x-n-2, x-n-4, ..., x-3n+2$ can be in $\V$, otherwise we would get a second representation. Therefore $x-3n\in\V$.
\end{case}
\end{proof}
The aforementioned analysis reveals two important facts. First, whether odd $u\in(2n,x)$ is in $\V$ depends only on the terms $u-2, u-4,..., u-2n$, and more precisely, $1(u)=1(u-2)+1(u-2n)$. Second, if we take $u=x-n$, we know by Lemma \ref{Right before x} that only $x-3n=(x-n)-2n$ is in $\V$ while $x-n-2,...,x-3n-2$ are not. We can encode this information in two isomorphic ways:
\begin{itemize}
\item as a binary sequence: $(1(x-3n),1(x-3n+2),...,1(x-n-2))=(1,0,...,0)$
\item as a polynomial $1+0\cdot t+...+0\cdot t^{n-1}=1$
\end{itemize}
We usually choose the polynomial encoding as it makes proofs neater, and it displays the connections between $\V$ and Pascal and Sierpinski's triangles. However, when convenient for the sake of the presentation, we will switch from one encoding to another by $(a_1,...,a_n)\leftrightarrow a_1+a_2 t+...+a_n t^{n-1}$.

Let $Q_1=1+0\cdot t+...+0\cdot t^{n-1}$, and let $Q_{j+1}=t\cdot Q_j$. Let $Q_j(i)$ denote the coefficient of $t_{i-1}$. Then $Q_{j}(i)=(1(x-3n+2i-2j))$. The relation $1(u)=1(u-2)+1(u-2n)$ imposes the condition $t^n=t+1$, meaning that we should consider $Q_j$ as polynomials in $\FF_2^n[t]/(t^n+t+1)$ as opposed to $\FF_2^n[t]$.
\begin{comment}
It turns out that the behavior of $\V$ is related to the behavior of a certain class of binary sequences exhibiting fractal behavior. We develop a theory of these sequences here to prove that whenever $n-1$ is a power of 2, $\V$ has a third even term $x=2n^2+2$. Define $Q_1=(1,0,...,0)\in\mathbb{F}_2^n$, and similarly define $Q_{k}\in\mathbb{F}_2^n$ by the following recursion:
\[
  Q_{k+1}(i)=\begin{cases}
               Q_k(i-1), &2\leqslant i\leqslant n\\
               Q_k(n)+Q_k(n-1), &i=1 \\
            \end{cases}
\]
If $\V$ has the third even term $x$, then $Q_1=(1(x-3n),1(x-3n+2),...,1(x-n-2))$ and $Q_{j}=(1(x-3n+2-2j),1(x-3n+4-2j),..., 1(x-n-2j))$.
\end{comment}

Now define $R_k:=Q_{1+n(k-1)}$. Thus $R_k(i)=1(x-(1+2k)n+2(i-1))$. $R_k$ stores information about which odd numbers immediately preceding $x-(2k-1)n$ are in the sequence.
\begin{comment}
$R_k$ satisfies the following recursion:
\[ R_{k+1}(i)=\begin{cases}
               R_k(i)+R_k(i-1), &2\leqslant i\leqslant n\\
               R_k(1)+R_k(n)+R_k(n-1), &i=1 \\
            \end{cases}
\] 
where we add them over characteristic 2. If $\V$ has the third even term $x$, then
$$R_k=(1(x-(1+2k)n),1(x-(1+2k)n+2),...,1(x-(2k-1)n-2))$$

where $1(y):=1_\V(y)$ is the indicator function of $\V$. 
\end{comment}
By the properties of $Q_k$, we have $R_1=1$ and $R_{k+1}=t^n R_k=(t+1)R_k$ where $R_k$ are polynomials over $\FF_2^n[t]/(t^n+t+1)$. In our discussions, we will however always take the representatives of $Q_k$ and $R_k$ with degree less than $n$, as these polynomials store the data that we care about.

We first give the basic properties of $R_k$:
\begin{lemma}\label{binary sequences, basic}
For $1\leqslant k\leqslant n-1$:
\begin{enumerate}
\item $R_k=(t+1)^k$
\item For $0\leqslant i\leqslant k$, $R_k(i)={{k-1}\choose{i-1}}\mod 2$
\item $R_k(1)=R_k(k)=1$
\item For $2\leqslant k\leqslant n-1$, $R_k(2)=R_k(k-1)=1$ precisely when $k$ is even.
\item $R_k(i)=0$ for $i>k$
\end{enumerate}
\end{lemma}

\begin{proof}
(1) follows immediately by induction on $k$ and the fact that $R_1=1$, $R_k=(t+1)R_1$. (2) results from the binomial theorem applied to $(t+1)^k$. (3) corresponds to the fact that ${{k-1}\choose{k-1}}={{k-1}\choose{0}}=1$ while (4) amounts to observing that for even $k$, ${{k-1}\choose{1}}=k-1$ is odd. Finally, (5) follows from the fact that $R_k$ has degree $k-1$.
\end{proof}
\begin{comment}
\begin{lemma}\label{binary sequences, basic}
If $1\leqslant k\leqslant n-1$, then $R_{k}(1)=R_k(k)=1$ and $R_k(i)=0$ for $i>k$. Moreover, for $2\leqslant k\leqslant n-1$, $R_k(k-1)=1$ precisely when $k$ is even.
\end{lemma}
    
\begin{proof}
For $k=1$, $R_k(k)=R_k(1)=Q_1(1)=1$, and $R_k(i)=Q_1(i)=0$ for $i>1$. Moreover, for $k=2$, $R_k(k-1)=R_2(1)=R_1(1)+R_1(n)+R_1(n-1)=1+0+0=1$.

Suppose the statement is true for some $1\leqslant k<n-1$. Then for $k+1$::
\begin{align*}
R_{k+1}(1) &=R_k(1)+R_k(n)+R_k(n-1)=1+0+0=1\\
R_{k+1}(k+1) &=R_k(k+1)+R_k(k)=0+1=1\\
R_{k+1}(i) &=R_{k}(i)+R_k(i-1)=0+0\; \rm{for}\; i>k+1 
\end{align*}
Moreover, $R_{k+1}((k+1)-1)=R_{k+1}(k)=R_k(k)+R_k(k-1)=1+R_k(k-1)$ is 1 iff $R_k(k-1)$ is 0. 
\end{proof}
\end{comment}

Translating the statement of Lemma \ref{binary sequences, basic} to statements about $\V$, we have the following:
\begin{lemma}
\ 

\begin{enumerate}
\item each of $x-n, x-3n,...,x-(2n-1)n$ is in $\V$
\item for all $1\leqslant k \leqslant n$, $x-(2n-2k+1)n-2k$ is in $\V$
\item $x-(2n-2k+1)n-2i\notin\V$ for $0<i<k$
\item $x-5n, x-9n+4, x-13n+8,...,x-(2n-1)n-6$ are all in $\V$ while $x-7n+2,x-11n+6,...,x-(2n-3)n-8$ are not
\end{enumerate}
\end{lemma}

\begin{proof}
All the statements follow from Lemma \ref{binary sequences, basic} and the fact that $$R_k(i)=1(x-(1+2k)n+2(i-1))$$
$R_k(1)=1$ corresponds to the fact that each of $x-n, x-3n,...,x-(2n-1)n$ is in $\V$. $R_k(k)=1$ means that for all $1\leqslant k \leqslant n$, $x-(2n-2k+1)n-2k$ is in $\V$. The assertion that $R_k(i)=0$ for $i>k$ implies that $x-(2n-2k+1)n-2i\notin\V$ for $0<i<k$. Finally, the fact that $R_k(k-1)=1$ precisely when $k$ is even tells us that $x-5n, x-9n+4, x-13n+8,...,x-(2n-1)n-6$ are all in $\V$ while $x-7n+2,x-11n+6,...,x-(2n-3)n-8$ are not. 
\end{proof}

The preliminary observations in Lemma \ref{binary sequences, basic} allow us to prove that the set $(R_k)_{k=1}^{n-1}$ is essentially the Sierpinski triangle, due to the fact that $R_k(i)={{k-1}\choose{i-1}}\mod 2$. This astonishing observation speaks volumes for the rigidity of the structure of $\V$. The set $(R_k)_{k=1}^{n-1}$ has thus a fractal behavior whose details we provide in the following lemma:

\begin{lemma}\label{fractal behavior of R}
Let $k\geqslant 0$ s.t. $2^{k+1}| n-1$. Then $R_{2^k+m}(2^k+i)=R_{2^k+m}(i)=R_{m}(i)$ for all $1\leqslant i,m\leqslant 2^k$. Moreover, $R_{2^k}(j)=1$ for $1\leqslant j\leqslant 2^k$.
\end{lemma}

%Before we prove Lemma \ref{fractal behavior of R}, we show $R_1$, ..., $R_9$ for $n=9$ to present the fractal behavior of $R_k$. We give $R_k$ in a vector form, where $(a_1,...,a_n)$ corresponds to the polynomial $a_1 + a_2 x + ...+ a_n x^{n-1}$:

Instead of proving Lemma \ref{fractal behavior of R}, which comes down to a well-established combinatorial fact that odd Newton symbols in Pascal's Triangle form a fractal, we show a picture $R_1$, ..., $R_{16}$ for $n=17$ to present the fractal behavior of $R_k$. Figure \ref{R_k vector} gives $(R_k)_{k=1}^{16}$ in a vector form, where $(a_1,...,a_n)$ corresponds to the polynomial $a_1 + a_2 t + ...+ a_n t^{n-1}$. Figure \ref{Sierpinski} shows the first 16 lines of Pascal's Triangle mod 2 which is related to $R_k$ by $R_k(i)={{k-1}\choose{i-1}}\mod 2$. 
\begin{comment}
\begin{align*}
R_1 &= (1,0,0,0,0,0,0,0,0)\\
R_2 &= (1,1,0,0,0,0,0,0,0)\\
R_3 &= (1,0,1,0,0,0,0,0,0)\\
R_4 &= (1,1,1,1,0,0,0,0,0)\\
R_5 &= (1,0,0,0,1,0,0,0,0)\\
R_6 &= (1,1,0,0,1,1,0,0,0)\\
R_7 &= (1,0,1,0,1,0,1,0,0)\\
R_8 &= (1,1,1,1,1,1,1,1,0)\\
R_9 &= (0,0,0,0,0,0,0,0,1)
\end{align*}
\end{comment}

\begin{figure}[h!]
\includegraphics[height = 0.25\textheight]{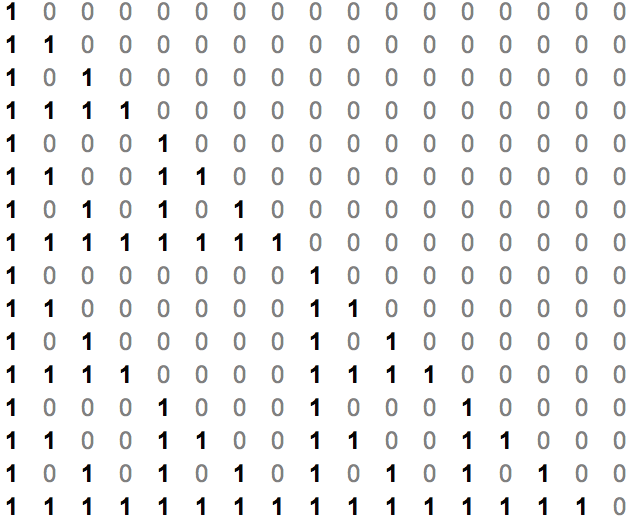}
\caption{$R_k$ for $1\leqslant k\leqslant 16=n-1$}
\label{R_k vector}
\end{figure}

\begin{figure}[h!]
\includegraphics[height = 0.25\textheight]{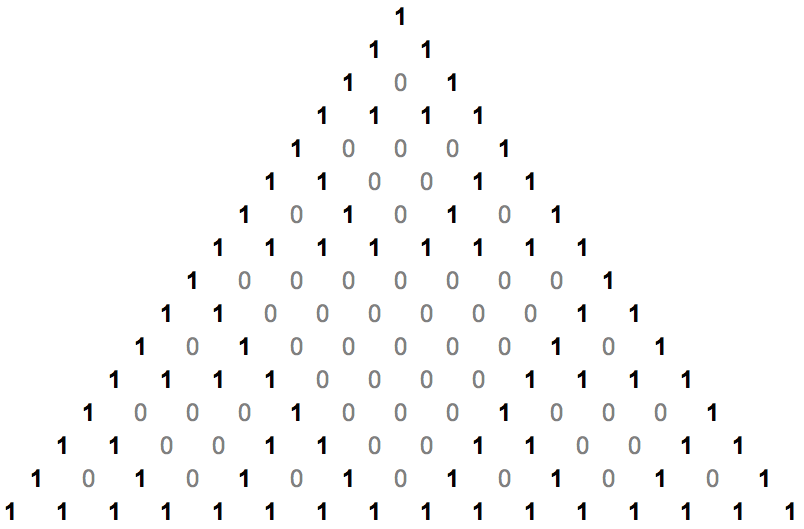}
\caption{First 16 lines of Pascal Triangle mod 2}
\label{Sierpinski}
\end{figure}

\begin{comment}

\begin{proof}
We will induct on $k$ and $m$. For $k=0$, we need to check that $R_1(1)=R_2(1)=R_2(2)$. $R_2(1)=R_1(1)+R_1(n)+R_1(n-1)=1+0+0=1$ and $R_2(2)=R_1(1)+R_1(2)=1+0=1$, hence the case of $k=0$ is true.

Suppose the statement is true for $k$. For $k+1$, $m=1$, we have the following:
\begin{align*}
R_{2^k+1}(1) &= R_{2^k}(1)+R_{2^k}(n)+R_{2^k}(n-1)=1+0+0=1 \\
R_{2^k+1}(j) &= R_{2^k}(j)+R_{2^k}(j-1)=1+1=0=R_1(j)\; \rm{for}\; 2\leqslant j\leqslant 2^k \\
R_{2^k+1}(2^k+1) &= R_{2^k}(2^k+1)+R_{2^k}(2^k)=0+1=1 \\
R_{2^k+1}(2^k+j) &= R_{2^k}(2^k+j)+R_{2^k}(2^k+j-1)=0+0\; \rm{for}\; j\geqslant2
\end{align*}
Hence the case is true for $m=1$. Suppose the statement is true for some $m-1$ s.t. $1<m<2^{k+1}$. Then we have the following:
\begin{align*}
R_{2^k+m}(j) &= R_{2^k+m-1}(j)+R_{2^k+m-1}(j-1)\\
			&=R_{m-1}(j)+R_{m-1}(j-1)\\
            &=R_m(j) \\
R_{2^k+m}(2^k+j) &= R_{2^k+m-1}(2^k+j)+R_{2^k+m-1}(2^k+j-1)\\
				&=R_{m-1}(j)+R_{m-1}(j-1)\\
				&=R_m(j) 
\end{align*}

In particular, $R_{2^{k+1}}(j)=R_{2^{k}+2^k}(j)=R_{2^k}(j)+R_{2^k}(2^k+j)=1+0=1$ for $1\leqslant j\leqslant 2^{k+1}$. 
\end{proof}
\end{comment}

One consequence of this lemma is that if $n-1=2^k$ for some $k\geqslant 2$, then $R_{n-1}=1+t+...+t^{n-2}=(1,...,1,0)$, i.e. all its coefficients but the last one are 0. We use this fact to prove that for these values of $n$, $\V$ has a third even term. 

Define the following polynomials in $\FF_2^n[t]/(t^n+t+1)$:
\begin{align*}
P_0 &= 1 + t + ... + t^{n-1}\\
P_1 &= 1 + t + ... + t^{n-2}\\
...\\
P_{k+1} &= (t^{n-1}+1)P_k \\
\end{align*}
These polynomials encode what happens at the beginning of the sequence. Writing them in the vector form, we get that:
\begin{align*}
P_0 &= (1(n),...,1(3n-2))=(1,...,1,1) \\
P_1 &= (1(n+2),...,1(3n))=(1,...,1,0) \\
...\\
P_k &= (1(n+2k),...,1(3n-2+2k))
\end{align*}

\begin{comment}
These binary sequences satisfy the following recurrence:

\[
  P_{k+1}(i)=\begin{cases}
               P_k(i+1), &1\leqslant i\leqslant n-1\\
               P_k(n)+P_k(1), &i=n
            \end{cases}
\]

Moreover, $P_{k+n}(i)=P_k(n)+\sum\limits_{j=1}^i P_k(j)$.
\end{comment}
Moreover, $P_{k+n}=(t+t^2+...+t^{n-1})P_k$. 

More importantly, there is the following relation between $P_k$'s and $Q_k$'s:
\begin{lemma}\label{P_k vs. Q_k}
$P_k=Q_{l+1}\iff P_{k+1}=Q_l$. Consequently, $P_{k+n}=R_l\iff P_k=R_{l+1}$.
\end{lemma}

\begin{proof}
The proof essentially follows from the fact that $t(t^{n-1}+1)=1$ in $\FF_2^n[t]/(t^n+t+1)$. Suppose $P_k=Q_{l+1}$. Then 
\begin{align*}
P_{k+1} &= (t^{n-1}+1)P_k\\
		&= (t^{n-1}+1)Q_{l+1}\\
        &= (t^{n-1}+1)t Q_l\\
        &= Q_l
\end{align*}
The second direction follows similarly. The other statement follows by induction.
\end{proof}

\begin{comment}
\begin{proof}
We will prove one direction; the other follows similarly. Suppose $P_k=Q_{k+1}$. For $1\leqslant i\leqslant n-1$: 
$$P_{k+1}(i)=P_k(i+1)=Q_{k+1}(i+1)=Q_k(i)$$
Moreover,
\begin{align*}
P_{k+1}(n)&=P_k(n)+P_k(1)\\
		&=Q_{k+1}(n)+Q_{k+1}(1)\\
        &=Q_k(n-1)+Q_k(n)+Q_k(n-1)\\
        &=Q_k(n)
\end{align*}
\end{proof}
\end{comment}

What this tells us is that the transformation taking $Q_k$ to $Q_{k+1}$ is an invertible operation whose inverse takes $P_k$ to $P_{k+1}$. This brings us to the proof of the first part of Theorem \ref{(2,n) theorem}.

\begin{lemma}\label{third even term}
Let $n-1=2^k$ for some $k\geqslant 2$. Then $x=2n^2+2$ is the third even element of $\V$. 
\end{lemma}

\begin{proof}
$$P_1=(1,...,1,0)=R_{2^k}=Q_{1+n(2^k-1)}=Q_{1+n(n-2)}$$
Inducting on Lemma \ref{P_k vs. Q_k}, we have that 
\begin{align*}
P_{1+n(n-2)}&=Q_1=(1,0,...,0)\\
			&=(1(2n^2-3n+2),...,1(2n^2-n))\\
            &=(1(x-3n),...,1(x-n-2))
\end{align*}
In particular, we should check whether $x-3n=2n^2+2-3n$, or $x=2n^2+2$, will work. We claim it does, i.e. $2n^2+2$ is, in fact, in $\V$. We know $2n^2+2=(2n^2-n+2)+n$ is a representation because 
\begin{align*}
P_{2+n(n-2)} &= (1+t^{n-1})P_{1+n(n-2)}\\
			 &= 1+t^{n-1}\\
             &= (1,0,...,0,1)\\
             &= (1(2n^2-3n+4),...,1(2n^2+2-n))
\end{align*}
and so $2n^2+2-n\in\V$. Suppose $2n^2+2$ has another representation. Then $$2n^2+2=(2n^2+2-(2l-1)n-2i)+((2l-1)n+2i)$$ for some $0\leqslant i<n$. We claim that it is not possible for both $2n^2+2-(2l-1)n-2i$ and $(2l-1)n+2i$ to be in $\V$. $2n^2+2-(2l-1)n-2i\in\V$ iff $R_l(n+1-i)=1$, and similarly $(2l-1)n+2i\in\V$ iff $R_{n-l}(i)=1$. Suppose they are both in $\V$. By Lemma \ref{binary sequences, basic}, this implies that $n+1-i\leqslant l$ and $i\leqslant n-l$ which is equivalent to saying that $n+1\leqslant i+l\leqslant n$, a contradiction. Hence $2n^2+2$ has no other representation, and so it is in $\V$. 
\end{proof}

We will now introduce tools necessary to prove that if $n-1$ is not a power of 2, then it has no third term $x$. Let $k$ be the largest positive integer s.t. $2^k|n-1$, and call $m=(n-1)/2^k$. Note that $k\geqslant 1$ because $n$ is odd and $m\geqslant 3$ because $n-1$ is not a power of 2. Define $S_l\in\mathbb{F}_2^n[t]/(t^n+t+1)$ the following way: 
\begin{align*}
S_1(t) &= t + t^{2^k+1} + t^{2\cdot 2^k + 1} + ... + t^{(m-1)2^k + 1}\\
... \\
S_{k+1}(t) &= t^n S_k(t)=(t+1)S_k(t)
\end{align*}

\begin{comment}
$S_1(i)=1$ precisely if $2^k|i-2$, and
\[
  S_{l+1}(i)=\begin{cases}
               S_l(i)+S_l(i-1), &2\leqslant i\leqslant n\\
               S_l(1)+S_l(n)+S_l(n-1), &i=1
            \end{cases}
\]
\end{comment}

Note that $S_l$ satisfies the same recurrence relation as $R_l$; moreover, $S_1(2^k j+i)=R_1(i)$ for $0\leqslant j\leqslant m-1$. It implies the following lemma:

\begin{lemma}\label{S_l}
For all $1\leqslant l\leqslant 2^k$, $1\leqslant i\leqslant 2^k$, $0\leqslant j\leqslant m-1$, the following is true:
\begin{enumerate}
\item $S_l(1)=0$ for $l < 2^k$
\item $S_l(1+s^k j+i)=R_l(i)$, where $R_l\in\mathbb{F}_2^{2^k}$
\item $S_{2^k}=1+t+...+t^{n-1}=(1)_{i=1}^n$
\end{enumerate}
\end{lemma}

The proof of this lemma mimics the proof of Lemma \ref{fractal behavior of R} and is omitted.
\begin{comment}
\begin{proof}
For $l=1$, $S_l(2)=S_l(2^k+2)=...=S_l(2^k(m-1)+2)=1=R_1(1)$, and $S_l(1+s^k j+i)=0=R_l(i)$ for $2\leqslant i\leqslant 2^k$, so the statement is true.

Suppose the statement is true for some $1\leqslant l\leqslant 2^k-2$. Then:

\begin{align*}
S_{l+1}(1)&=S_l(1)+S_l(n)+S_l(n-1)=0+S_l(2^k m+1)+S_l(2^k m)\\
&=R_l(2^k)+R_l(2^k-1)=0+0=0\\
S_{l+1}(1+2^k j+i) &= S_l(1+2^k j+i) + S_l(1+ 2^k j + (i-1))\\
&=R_l(i)+R_l(i-1)=R_{l+1}(i)\; \rm{for}\; i\geqslant 2\\
S_{l+1}(1+2^k j+1) &= S_l(1+2^k j+1)+S_l(1+2^k j)=R_l(1)+R_l(n)\\
&=R_l(1)+R_l(n)+R_l(n-1)=R_{l+1}(1)\; \rm{because}\; R_l(n-1)=0
\end{align*}
So the statement is true for $l\leqslant 2^{k-1}$. For $l=2^k$,
\begin{align*}
S_{2^k}(1) &= S_{2^{k}-1}(1) + S_{2^k-1}(n) +S_{2^k-1}(n-1)\\
		   &= 0 + S_{2^k-1}(2^k m + 1) + S_{2^k - 1}(2^k m)\\
      	   &= R_{2^k-1}(2^k) + R_{2^k -1}(2^k-1)\\
           &= 0 +1 = 1\\
S_{2^k}(1+2^k +i) &= S_{2^k-1}(1+2^k j+i)+S_{2^k-1}(2^k j + i)\\
				  &= R_{2^k-1}(i)+R_{2^k - 1}(i-1)=R_{2^k}(i)\; \rm{for}\; i\geqslant 2\\                  
S_{2^k}(1+2^k +1) &= S_{2^k-1}(1+2^k j+1)+S_{2^k-1}(1+2^k j)\\
				  &= R_{2^k-1}(1)+R_{2^k - 1}(2^k)\\
                  &= R_{2^k-1}(1) = 1  
\end{align*}

Hence $S_{2^k}=(1)_{i=1}^{2^k}$.
\end{proof}
\end{comment}
Note that $P_0=(1(n),...,1(3n-2))=(1,...,1)=S_{2^k}$. Since $S_l$ is defined by the same recurrence as $R_l$, we can apply Lemma \ref{P_k vs. Q_k} again to see that $S_{2^k-1}=P_{n}$, $S_{2^k-2}=P_{2n}$, ..., $S_1=P_{(2^k-1)n}$.

\begin{lemma}\label{P_{2^k n}}
$$P_{2^k n}=(t+t^2+...+t^{2^k})+(t^{2\cdot 2^k+1}+...+t^{3\cdot 2^k})+...+ (t^{(m-1)\cdot 2^k+1}+...+t^{m\cdot 2^k})$$ In particular, $P_{2^k n}(i)=1$ for $i\geqslant n-2^k + 1$.
\end{lemma}

\begin{proof}
\begin{align*}
P_{2^k n} &= (t+t^2+...t^{n-1})P_{(2^k-1)n}\\
		  &= (t+t^2+...t^{n-1})S_1\\
          &= (t+t^2+...t^{n-1})\sum\limits_{i=0}^{m-1}t^{i 2^k+1}\\
          &= \sum\limits_{i=0}^{m-1}[(t^{i 2^k +2}+t^{i 2^k+3}+...+t^{m 2^k})+(t^{m 2^k +1}+t^{m 2^k +2}+...+t^{(m+i)2^k+1})]\\
          &=\sum\limits_{i=0}^{m-1}[(t^{i 2^k +2}+t^{i 2^k+3}+...+t^{m 2^k})+(1+t)(1+t+...+t^{i 2^k})]\\
          &=\sum\limits_{i=0}^{m-1} [1+(t^{i 2^k +1}+t^{i 2^k +2}+...+t^{m 2^k})]\\
          &=\sum\limits_{i=0}^{m-1} (i+1)(t^{i 2^k +1}+t^{i 2^k +2}+...+t^{m 2^k})\\
          &=(t+t^2+...+t^{2^k})+(t^{2\cdot 2^k+1}+...+t^{3\cdot 2^k})+...+ (t^{(m-1)\cdot 2^k+1}+...+t^{m\cdot 2^k})
\end{align*}
\end{proof}

\begin{comment}
\begin{proof}
Using the fact that $P_{l+n}(i)=P_l(n)+\sum\limits_{j=1}^i P_l(j)$, we can observe that 
$$P_{2^k n}(i)=P_{(2^k-1) n}(n)+\sum\limits_{j=1}^i P_{(2^k -1) n}(j)=S_1(n)+\sum\limits_{j=1}^i S_1(j)$$
We know $2^k\nmid j-2$ for $2^k(m-1)+2=n-2^k+1<j \leqslant n$, hence $S_1(j)=0$ for these values of $j$. In particular, this implies that
\begin{align*}
P_{2^k n}(n-2^k)&=P_{2^k n}(n-2^k+1)=...=P_{2^k n}(n)\\
&=\sum\limits_{j=1}^{n-2^k+1}S_1(j)
\end{align*}

Here is where we are using the fact that $n-1$ is not a power of 2. $S_1(j)=1$ for precisely $m$ values of $j$ (specifically, 2, $2^k+2, 2^k\times 2+2$, $2^k\times 3+2$,...,$2^k\times(m-1)+2=n-2^k+1$). Since $m$ is odd, we have 
$$P_{2^k n}(i)= \sum\limits_{j=1}^{n-2^k+1}S_1(j)=m=1$$
for desired values of $i$. 
\end{proof}
\end{comment}

In particular, $1((2^{k+1}+3)n-2^{k+1}) = P_{2^k n}(n-2^k)=1$, hence $(2^{k+1}+3)n-2^{k+1}$ is in $\V$. We also have $x-(2^{k+1}+3)n+2^{k+1}\in\V$ because $1(x-(2^{k+1}+3)n+2^{k+1})=R_{2^k+1}(2^k+1)=R_1(1)=1$. Hence $x$ has a second representation
$$x=((2^{k+1}+3)n-2^{k+1})+ (x-(2^{k+1}+3)n+2^{k+1}) $$
and so we arrive at contradiction. Thus $\V$ does not have a third even element when $n-1$ is not a power of 2. 

\end{proof}

\section{Proof of Theorem 1.3.}
Throughout this section, let $\V:=\V(a,b)$ be the $\V$-sequence generated by relatively prime $a$ and $b$ s.t. $a$ is even and $b>2a$ is odd. These sequences are conjectured to be regular (see Conjecture \ref{Conjecture about Regular Modified Ulam Sequences}). We will relate $\V$ to a subset of $\ZZ_{\geqslant 0}^2$ via Freiman homomorphism, and then we will provide a lower bound on the number of even terms that these sequences have.
Let $\W$ be the set of points in $\ZZ_{\geqslant 0}^2$ defined recursively in the following way:
\begin{enumerate}
\item $\W$ contains $(1,0), (0,1)$
\item $(2,0)$ is not in $\W$
\item Each subsequent vector $z$ in $\W$ is the smallest vector in $\ZZ_{\geqslant 0}^2\setminus{\{(2,0)\}}$ for which there is a unique pair of distinct $x,y$ already in $\W$ s.t. $x+y=z$
\end{enumerate}

$\W$ is thus defined almost the same way as $\V((1,0),(0,1))$ except that the vector $(2,0)$ is specifically excluded from the set. We can determine the structure of $\W$ explicitly:
\begin{lemma}\label{asymetric Ulam}
$\W$ consists precisely of the following points:
\begin{itemize}
\item $(1,0)$
\item $(n,1)$ for all $n\in\ZZ_{\geqslant 0}$
\item $(n,3)$ for even $n$
\item $(3,5)$
\item $(0,n)$ for $n\in\V(1,2)$
\item $(1,n)$ for $n=1$ or $4|n$
\item $(2,n)$ for $n\in 1+8\ZZ_{\geqslant 0}$, $n\in 3+8\ZZ_{\geqslant 0}$
\item $(n,m)$ for odd $n\geqslant 5$ and $m=1\; \rm{mod}\; 4$, $m\geqslant 5$
\end{itemize}
\end{lemma}
\begin{proof}
The proof is an unpleasant exercise in induction and case exhaustion but presents no substantial difficulties and requires no interesting ideas. Instead of presenting it, we provide a picture of $\W$ in Figure \ref{nonsymmetric}. 
\end{proof}

	\begin{figure}[h!]
	\includegraphics[height = 0.2\textheight]{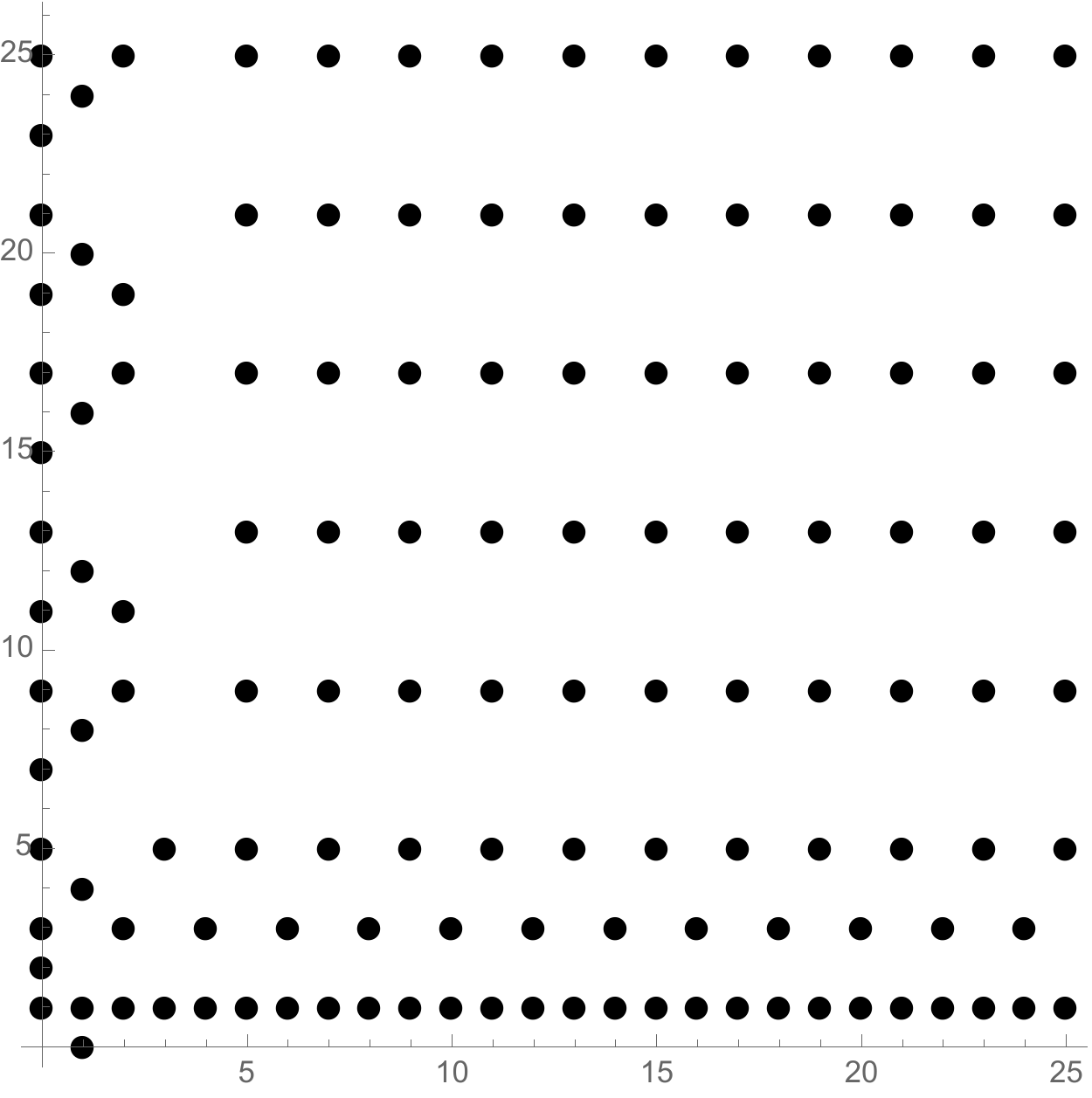}
	
	\caption{The structure of $\W$.}
	\label{nonsymmetric}
	\end{figure}

The main point of this section is to show that $\V\cap [0,ab]$ is structurally equivalent to a subset of $\W$, and hence we can deduce information on the structure of $\V$ from the structure of $\W$.

\begin{definition}\label{Freiman homomorphism for modified Ulam}
\begin{align*}
f:\V &\to \mathbb{Z}^2/(-b,a) \\
a &\mapsto (1,0)+\mathbb{Z}(-b,a)\\
b &\mapsto (0,1)+\mathbb{Z}(-b,a)\\
n+m& \mapsto f(n)+f(m)+\mathbb{Z}(-b,a), n,m\in \V
\end{align*}
For simplicity, we write $(c,d):=(c,d)+\mathbb{Z}(-b,a)$.
\end{definition}

\begin{lemma}
$f$ is a well-defined Freiman homomorphism of order 2.
\end{lemma}
\begin{proof}
Let $z_1,z_2,z_3,z_4\in \V$ s.t. $f(z_i)=(x_i,y_i)$, i.e. $z_i=a x_i+b y_i$, and suppose that $z_1+z_2=z_3+z_4$. We need to show that $f(z_1)+f(z_2)=f(z_3)+f(z_4)$. Note that 
\begin{align*}
0=z_1+z_2-z_3-z_4=a(x_1+x_2-x_3-x_4)+b(y_1+y_2-y_3-y_4)
\end{align*}
Since $a|(y_1+y_2-y_3-y_4)$, we have that 
\begin{align*}
f(z_1)+f(z_2)-f(z_3)-f(z_4)&=(x_1+x_2-x_3-x_4, y_1+y_2-y_3-y_4)\\
&=(-\frac{b}{a}(y_1+y_2-y_3-y_4), y_1+y_2-y_3-y_4)
\end{align*}
By assumption, $c:=(y_1+y_2-y_3-y_4)/a\in\mathbb{Z}$, hence 
\begin{align*}
f(z_1)+f(z_2)-f(z_3)-f(z_4)=(-bc,ac)=c(-b,a)=0
\end{align*}
Thus $f$ is a well-defined Freiman homomorphism of order 2. 
\end{proof}
This implies in particular that we can extend $f$ to $\V+\V$, where $f(n+m)=f(n)+f(m)$ for all $n,m\in \V$. From now on, $f$ refers to this extended map from $\V+\V$ to $\mathbb{Z}^2/(-b,a)$.

\begin{lemma}
$f|_\V$ is injective.
\end{lemma}
\begin{proof}
Let $z_1,z_2\in \V$ be such that $f(z_1)=f(z_2)$. That is to say, if $f(z_1)=(x_1,y_1), f(z_2)=(x_2,y_2)$, then $x_1=x_2-na$ and $y_1=y_2+nb$ for some $n\in\mathbb{Z}$. Therefore 
\begin{align*}
z_1-z_2&=ax_1+by_1-ax_2-by_2\\
&=a(x_1-x_2)+b(y_1-y_2)\\
&=a(-nb)+b(na)=0
\end{align*}
Hence $f$ is injective.
\end{proof}

Consider the box $I=\{(i,j):0\leqslant i < b, 0\leqslant j <a\}\subset\mathbb{Z}^2$. Each element of this box is a representative of a distinct coset of $\mathbb{Z}^2/(-b,a)$. Hence we can view $f|_{f^{-1}(I)}$ as an injective map from $\V+\V$ to $I$. Since $f|_{f^{-1}(I)}$ sends generators of $\V$ to generators of $\W$ and preserves the additive structure of both sets, we have that $\V\cap f^{-1}(I)$ is structurally equivalent to $\W\cap I$. Moreover, $(\V+\V)\cap f^{-1}(I)$ is structurally equivalent to $(\W+\W)\cap I$. What this means is that initially, the set $\V(a,b)$ behaves as $\W$. This gives us the precise structure of $\V(a,b)\cap [0,ab)$.
As a side note, $ab\notin \V$ because $ab$ has no representation. Therefore we know the structure of $\V(a,b)\cap[0,ab]$.

\begin{theorem}
$\V(a,b)\cap [0,ab]$ consists of numbers of the following form which are no greater than $ab$:
\begin{itemize}
\item $a$
\item $b+na$ for $n\in\ZZ_{\geqslant 0}$
\item $3b+na$ for even $n$
\item $5b+3a$
\item $nb$ for $n\in\V(1,2)$
\item $nb+a$ for $n=1$ or $4|n$
\item $nb+2a$ for $n\in 1+8\ZZ_{\geqslant 0}$, $n\in 3+8\ZZ_{\geqslant 0}$
\item $mb+na$ for odd $n\geqslant 5$ and $m=1\; \rm{mod}\; 4$, $m\geqslant 5$\end{itemize}
\end{theorem}
\begin{proof}
The proof follows directly from Lemma \ref{asymetric Ulam} and the structural equivalence.
\end{proof}
\begin{corollary}
The only even terms of $\V(a,b)$ less than $ab$ are $a, 2b$ and $nb+a$ for $n\in 4\ZZ_{\geqslant 0}\cap [0,a)$.
\end{corollary}

\begin{corollary}
$\V(a,b)$ has at least $2+\lfloor{\frac{a}{4}}\rfloor$ even terms.
\end{corollary}

These are however not all the even terms that $\V(a,b)$ can have. The following conjecture, based entirely on numerical data, enumerates the even terms of some modified Ulam sequences.

\begin{conjecture}\label{even terms of regular modified Ulam sequences}
For the following values of $a,b$, where $a,b$ are relatively prime, these are all the even terms of $\V(a,b)$:
\begin{itemize}
\item $\V(4,b)$, $b>8$ odd: $4, 2 b, 4 b + 4, 12 b - 4$
\item $\V(6,b)$, $b>12$ odd: $6, 2 b, 4b+6$ 
\item $\V(8,b)$, $b>16$ odd: $8, 2 b, 4 b + 8, 8 b + 8, 24 b - 8$
\item $\V(10,b)$, $b>20$ odd: $10, 2 b, 4 b + 10, 8 b + 10$
\item $\V(12,b)$, $b>24$ odd: $12, 2 b, 4 b + 12, 8 b + 12, 12 b + 12, 20 b + 12, 28 b - 12, 
38 b + 48$
\item $\V(14,b)$, $b>28$ odd: $14, 2 b, 4 b + 14, 8 b + 14, 12 b + 14, 22 b + 14, 30 b - 14$
\item $\V(16,b)$, $b>32$ odd: $16, 2 b, 4 b + 16, 8 b + 16, 12 b + 16, 16 b + 16, 24 b + 16, 
32 b + 16, 44 b - 16, 44 b + 16$
\end{itemize}
\end{conjecture}
% We could enumerate the conjectured even terms of $\V(a,b)$ for $a$ even, $b>2 a$ odd and relatively prime to $a$ for other values of $a$ as well; it is also easy to prove that $\V(a,b)$ would contain these numbers. The hard part is to show that there are no more even terms after the conjectured ones. Using techniques similar to those in ... (OUR PAPER), one could prove that given specific values of $a,b$, the sequence $\V(a,b)$ has finitely many even terms. If a rigidity conjecture similar to one for Ulam sequences holds, then one could hope to use it to generalize such a result to $\V(a,c)$ for all $c=b\; \rm{mod}\; L$, where $L$ is a multiple of $a$. 
 
Note that the even terms of $\V(a,b)$ for $a>2$ are conjectured to be linear in $b$. This is strikingly different from the case of $\V(2,b)$; for $\V(2,b)$ s.t. $b-1$ is a power of $2$, $\V(2,b)$ has a third even term which is quadratic in $b$. For Ulam sequences, a similar phenomenon is only conjectured for $\U(4,b)$ where $b+1$ is a power of 2: in this case, $\U(4,b)$ is conjectured to have 4 even terms: $4, 2b+4, 4b+4, 4b^2+2b-4$ \cite{finch_1992_1, cassaigne_finch_1995}.

\section{Freiman homomorphisms between Ulam sets}

%The method of finding a Freiman homomorphism between a sequence and a two-dimensional set can be applied to other cases as well. Using the same argument, one shows that for $a,b$ relatively prime s.t. $b<2a$, $\V(a,b)\cap[0,ab]$ is structurally equivalent to $\V((1,0),(0,1))\cap([0,b)\times [0,a))$ by the map $f: a\mapsto (1,0), b\mapsto (0,1)$. 

We can extend the technique in the previous section to the Ulam sets. For the entirety of this section, let $a<b$ be relatively prime, and let $\U:=\U(a,b)$ and $\W:=\U((1,0),(0,1))$ be the Ulam set generated by $(1,0), (0,1)$. Recall the structure of $\W$ from Kravitz and Steinerberger \cite{kravitz_steinerberger_2017}, presented in Figure \ref{U1001}:
\begin{lemma}\label{Structure of U((1,0),(0,1))}
$\W$ consists of the following vectors:
\begin{itemize}
\item $(n,1)$, $n\in\ZZ_{\geqslant 0}$
\item $(1,m)$, $m\in\ZZ_{\geqslant 0}$
\item $(n,m)$, $n,m \geqslant 3$ odd
\end{itemize}
\end{lemma}

	\begin{figure}[h!]
	\includegraphics[height = 0.2\textheight]{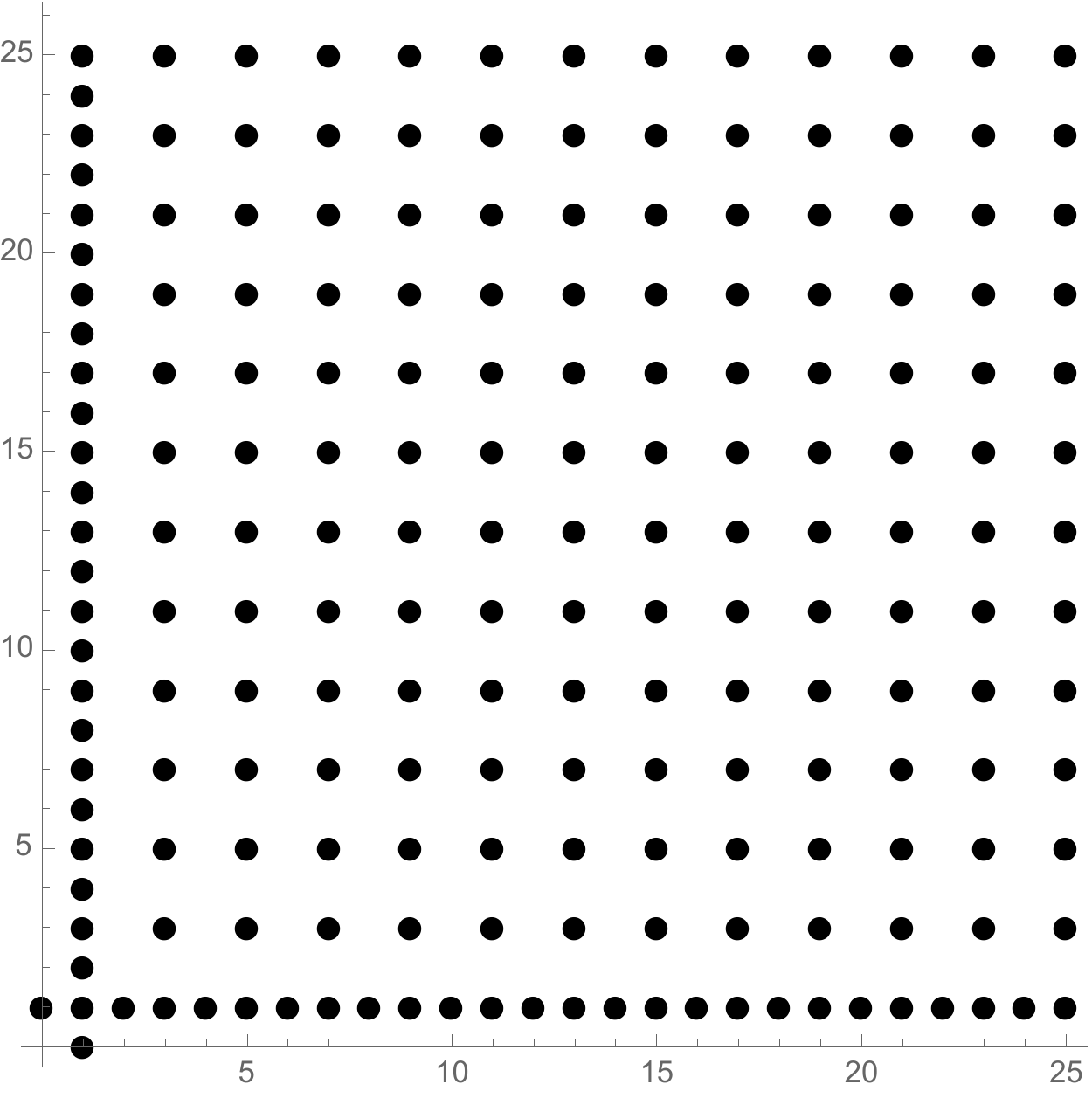}
	
	\caption{$\U((1,0),(0,1))$}
	\label{U1001}
	\end{figure}

By finding a connection between $\U$ and $\W$, we will prove the following result about arithmetic progressions of certain common differences in Ulam sequences:

\begin{theorem} \label{arithmetic progressions in Ulam sequences}
We have the following:

\begin{itemize}
\item for any $c>(b+1)a$, $\{c, c+a, ..., c+ba\}\subset \U\implies c+(b+1)a\notin \U$
\item for any $c>b(a+1)$, $\{c, c+b, ..., c+ab\}\subset \U\implies c+(a+1)b\notin \U$
\end{itemize}

In other words, no $b+2$ Ulam numbers are in an arithmetic progression of common difference $a$, and no $a+2$ Ulam numbers are in an arithmetic progression of common difference $b$.
\end{theorem}

We first need to define an appropriate Freiman homomorphism.
\begin{definition}\label{Freiman homomorphism for Ulam}
\begin{align*}
f:\U &\to \mathbb{Z}^2/(-b,a) \\
a &\mapsto (1,0)+\mathbb{Z}(-b,a)\\
b &\mapsto (0,1)+\mathbb{Z}(-b,a)\\
n+m& \mapsto f(n)+f(m)+\mathbb{Z}(-b,a), n,m\in \V
\end{align*}
For simplicity, we write $(c,d):=(c,d)+\mathbb{Z}(-b,a)$.
\end{definition}
Note that this is the same definition as Definition \ref{Freiman homomorphism for modified Ulam}. We begin by stating several lemmas about $f$ which are proved the same way as similar lemmas in the previous section.

\begin{lemma}
f is a well-defined Freiman homomorphism of order 2.
\end{lemma}

This implies in particular, that we can extend $f$ to $\U+\U$, where $f(n+m)=f(n)+f(m)$ for all $n,m\in \U$. From now on, $f$ refers to this extended map from $\U+\U$ to $\mathbb{Z}^2/(-b,a)$.

\begin{lemma}
$f|_\U$ is injective.
\end{lemma}

Consider the box $I=\{(i,j):0\leqslant i<b, 0\leqslant j<a\}\subset\mathbb{Z}^2$. Like in the previous section, there is a structural equivalence between $\U\cap f^{-1}(I)$  and $\W\cap I$, and between $(\U+\U)\cap f^{-1}(I)$ and $(\W+\W)\cap I$.   

\begin{lemma}\label{initial terms of Ulam on coordinate axes}
If $a\neq 1$, $2\leqslant k\leqslant b$, $2\leqslant l\leqslant a$, then $ka, lb\notin \U$.
\end{lemma}

\begin{proof}
$(1,0)\in \W$ but $(k,0)\notin \W$ for $2\leqslant k<b$. By the structural equivalence
of $\U\cap f^{-1}(I)$ and $\W\cap I$, we have that $ka\notin \U$ for $2\leqslant k<b$. Likewise, $(0,1)\in\W$ but $(0,l)\notin \W$ for $2\leqslant l<a$, and the structural equivalence implies that $lb\notin \U$ for $2\leqslant l<a$. It remains to show that $ba\notin \U$. More precisely, we will show that $ba$ has no Ulam representation. Suppose $ba$ has an Ulam representation $ba=(ca+db)+(ea+fb)$. Taking this mod $b$, we get that $ba\equiv(c+e)a$, hence $b|(c+e-b)$. If $c+e=b$, then $ba=(b-c)a+ca$, but at least one of $(b-c)a, ca$ is not in $\U$ by the previous argument. Hence $c+e\geqslant b+b$, and so $(ca+db)+(ea+fb)\geqslant(b+b)a$ which is a contradiction. Hence $ba$ has no Ulam representation, and is not in $\U$.
\end{proof}
\begin{lemma}\label{initial terms of Ulam sequence}
If $0<i<a$, $0<j<b$, then $ja+ib\in \U$ if and only if $j=1$ or $i=1$ or both $i,j$ odd. 
\end{lemma}
\begin{proof}
This is a direct consequence of the structural equivalence of $\U\cap f^{-1}(I)$ and $\W\cap I$. Since the only elements $(i,j)$ of $\W$ are $(1,0), (0,1), (1,j), (i,1), (2k+1, 2l+1)$, we get our result. 
\end{proof}

\begin{lemma} \label{a(b+1), (a+1)b}
$a(b+1), (a+1)b\in \U$.
\end{lemma}
\begin{proof}
Note first that $a(b+1)=b+[(a-1)b+a]$. We know that $(a-1)b+a\in \U$ by Lemma \ref{initial terms of Ulam sequence}, hence this is a valid Ulam representation of $a(b+1)$. Suppose there exists another Ulam representation $a(b+1)=(ca+db)+(ea+fb)$. Taking the equation mod $a$ and mod $b$, we get respectively that $0\equiv(d+f)q$ mod $a$ and $a\equiv(c+e)a$ mod $b$. This requires that $a|(d+f)$. If $d+f\geqslant a$, then it is necessary that $d+f=a$ and $ca+ea=a$. Then we have that $a(b+1)=(a+db)+(a-d)b$. By Lemma \ref{initial terms of Ulam on coordinate axes}, the only element of the form $(a-d)b$ is $b$, and so we get the aforementioned representation. If $d+f=0$, then $a(b+1)=ca+ea$. This means that at least one of $c,e$ (say, $c$) is in the interval $2\geqslant c\geqslant b$, and we know from Lemma \ref{initial terms of Ulam on coordinate axes}. that this will not be an Ulam number. Hence the only Ulam representation of $a(b+1)$ is $a(b+1)=b+[(a-1)b+a]$, and so $a(b+1)\in \U$. A similar argument shows that $b(a+1)\in \U$
\end{proof}
 
Now we are ready to prove Theorem \ref{arithmetic progressions in Ulam sequences}.

\begin{proof}[Proof of Theorem \ref{arithmetic progressions in Ulam sequences}]
Suppose that we have $c, c+a, c+2a, ..., c+ba\in \U$, $c>(b+1)a$. Then $c+(b+1)a$ has two Ulam representations: $c+(b+1)a=(c+ba)+a=c+[(b+1)a]$. Hence it is not in $\U$.
Similarly, suppose that we have $c, c+b, c+2b, ..., c+ab\in \U$, $c>(a+1)b$. Then $c+(a+1)b$ has two Ulam representations: $c+(a+1)b=(c+ab)+b=c+[(a+1)b]$. Hence it is not in $\U$, and the theorem is proved.
\end{proof}

Note that the argument of the last paragraph can be generalized: if $x, nx\in \U$ for some positive integer $n>1$, then $\U$ has no arithmetic progressions of common difference $x$ and size $n$ if the first element of the progression is greater than $nx$ because if $c, c+x, ..., c+(n-1)x$, then $c+nx$ has two Ulam representations. In particular, we can prove the following statement:
\begin{theorem}
Let $\U=\U(a,b)$, $gcd(a,b)=1$, $a\geqslant 3$. Then $\U$ contains no arithmetic progression of common difference $a+b$ and size $>3$.
\end{theorem}
\begin{proof}
Note that for any values of $a,b$ we have $a+b\in \U$. It is enough to show that $3a+3b\in \U$ because if $c, c+(a+b), c+2(a+b), 3a+3b\in U$, then $c+3(a+b)=[c+2(a+b)]+(a+b)=c+(3a+3b)$ has two Ulam representations. For $a>3$, we have that $(3,3)\in \W$, and the structural equivalence implies that $3a+3b\in \U$.

We claim that for $a=3$, we still have $3a+3b\in U$. By structural equivalence and Lemma \ref{Structure of U((1,0),(0,1))}, we know that the initial terms of $\U$ include $a,b,a+b, 2a+b, a+2b$, but they do not include $2a, 2b, 2a+2b$. Clearly, $3a+3b$ has a representation $3a+3b=(2a+b)+(a+2b)$, and we claim that this is its only Ulam representation. Suppose it has another representation $3a+3b=(ca+db)+(ea+fb)$. Taking this mod $b$, we get that $3a=(c+e)a$, hence $b|c+e-3$. If $c+e=3$, then $d+f=3$. The only representation satisfying these conditions is $(2a+b)+(a+2b)$. If by contrast $c+e\geqslant b+3$, then $(ca+db)+(ea+fb)\geqslant (b+3)a+(d+f)b=ab+3a+(d+f)b$. This necessitates that $3=a+d+f$, hence $d=f=0$. then $3a+3b=ca+(b-c)a$. Since $b>3$, at least one of $c, b-c$ is less than $b$. Without the loss of generality, suppose $c<b$. Then $ca\notin \U$, hence it cannot be a summand of $3a+3b$, and so $3a+3b$ has no other Ulam representation in this case. Thus $3a+3b$ has exactly one Ulam representation, hence it is in $\U$, and so the theorem is proved.
\end{proof}
The method used in this proof can be further generalized:
\begin{theorem}
If $c,d$ are odd positive integers such that $1\leqslant c<\frac{a}{3}$, $1\leqslant d<\frac{b}{3}$, then there is no arithmetic progression of common difference $da+cb$ and size greater than $3$.
\end{theorem}
\begin{proof}
Because of the assumptions made, $3da+3cb$ is in $\U$, and the theorem follows from the same argument as used in the previous theorems.
\end{proof}

The structural equivalence, together with Lemma \ref{Structure of U((1,0),(0,1))}, helps us deduce some of the even terms of Ulam sequences conjectured to be regular.
\begin{theorem}
Let $\U(a,b)$ be an Ulam sequence with $a,b$ relatively prime, and one of $a,b$ even. Then:
\begin{itemize}
\item if $a$ is even, then $\U(a,b)\cap[0,ab]\cap 2\ZZ_+=\{a, a+2b, a+4b, ...,a+(a-2)b\}$. Moreover, by Lemma \ref{a(b+1), (a+1)b}, $a+ab\in\U$.
\item if $b$ is even, then $\U(a,b)\cap[0,ab]\cap 2\ZZ_+=\{b, b+2a, b+4a,...,b+(b-2)a\}$. Moreover, by Lemma \ref{a(b+1), (a+1)b}, $b+ba\in\U$.
\end{itemize}
\end{theorem}

\begin{corollary}
Let $\U(a,b)$ be an Ulam sequence with $a,b$ relatively prime, and one of $a,b$ even. Then:
\begin{itemize}
\item if $a$ is even, then $\U(a,b)$ has at least $1+a/2$ even terms.
\item if $b$ is even, then $\U(a,b)$ has at least $1+b/2$ even terms.
\end{itemize}
\end{corollary}

The following conjecture comes from Finch. If it is true, it shows that we hit almost all the even terms of $\U(a,b)$:
\begin{conjecture}\cite{finch_1992_1}
Let $\U(a,b)$ be an Ulam sequence with $a,b$ relatively prime, and one of $a,b$ even. Then:
\begin{itemize}
\item if $a$ is even, then $\U(a,b)$ has $2+a/2$ even terms, and these are precisely $\{a, a+2b,...,a+ab, (2a+4)b\}$.
\item if $b$ is even, then $\U(a,b)$ has $2+b/2$ even terms, and these are precisely $\{b, b+2a,...,b+ba, (2b+4)a\}$.
\end{itemize}
\end{conjecture}

No comparably simple conjecture has been found for modified Ulam sequences $\V(a,b)$ where $a,b$ are relatively prime, $a$ is even, and $b>2a$ is odd. The patterns of even terms for various values of $a$ seem  erratic, see Conjecture \ref{even terms of regular modified Ulam sequences}.

\section{Some Empirical Observations about (1,1,1)-sequences}

We will now look at the behavior of $\Z_{(1,1,1)}$, also known as $(1,3)$-additive sequence. They have already been investigated by Finch \cite{finch_1992_1}. We summarize his results and our observations, as well as present open questions concerning these sequences. First we introduce terminology that we use while describing these sequences. Given a regular sequence $(a_n)$, the period $N$ is the smallest positive integer s.t. $a_n=a_{n+N}$ for all $n>n_0$ for some $n_0\in\NN$. The quantity $D=a_{n+N}-a_n$ for $n>n_0$ is called the fundamental difference, and it is easy to see that the density of $(a_n)$ is $N/D$. Given $N$ and $n_0$, we take the sequence of differences to be 
$$P=(a_{n_0+1}-a_{n_0}, a_{n_0+2}-a_{n_0+1},...,a_{n_0+N}-a_{n_0+N-1})$$ for $n>n_0$. Then the sequence $(a_{n+1}-a_n)_{n=n_0}^\infty$ is a union of copies of $P$. Note that $P$ depends on the choice of $n_0$ - if we pick two different $n_0$'s, then the resulting sequences of differences will differ by a shift. 

Finch proved that the following $(1,1,1)$-sequences are regular \cite{finch_1992_1}: 
\begin{theorem}[Finch]
The following sequences are regular:
\begin{enumerate}
\item $\Z_{(1,1,1)}(1,2,w)$ is regular with $w = 3\mod 6$ for $w>45$. Moreover, $N=\frac{1}{3}(7w+9)$ and $D=21w+1$.
\item $\Z_{(1,1,1)}(1,2,w)$ is regular with $w = 0\mod 6$ for $w>24$. Moreover, $N=w+1$ and $D=7w+1$.
\item $\Z_{(1,1,1)}(1,3,w)$ is regular with $w = 0\mod 2$ for $w>22$. Moreover, $N=w+1$ and $D=7(w+1)$.
\item $\Z_{(1,1,1)}(1,3,w)$ is regular with $w = 1\mod 4$ for $w>17$. Moreover, $N=\frac{1}{4}(w+3)$ and $D=5w+9$.
\end{enumerate}

\end{theorem}

We have observed that all $\Z_{(1,1,1)}(1,2,w)$ are regular for $w = 0\mod 3$ except $w=6$. The details of their period, sequence of differences, and fundamental difference are provided below:

\begin{observation}
\
\begin{itemize}
\item $w=3$: $N=3$, $D=25$, $P=(1,2,22)$.
\item $w=9$: $N=86$, $D=572$.
\item $w=12$: $N=112$, $D=760$.
\item $w=15$: $N=16$, $D=106$.
\item $w=18$: $N=206$, $D=1394$.
\item $w=21$: $N=52$, $D=442$.
\item $w=24$: $N=665$, $D=3581$.
\item $w=27$: $N=47$, $D=378$.
\item $w=33$: $N=80$, $D=694$.
\item $w=39$: $N=40$, $D=274$. 
\item $w=45$: $N=46$, $D=316$.
\end{itemize}
\end{observation}

This merits some immediate observations:
\begin{observation} \label{various (1,1,1)-sequences}
\
\begin{enumerate}
\item The formulas for period and fundamental difference for $w=21, 33$ follow the same rule as the formulas for $w>45$, $w = 3 \mod 6$ proved by Finch. However, if one looks more closely at how these sequences evolve, the evolution of $\Z_{(1,1,1)}(1,2,21)$ and $\Z_{(1,1,1)}(1,2,33)$ differs from the evolution of $\Z_{(1,1,1)}(1,2,w)$ for $w>45$, $w = 3 \mod 6$.
\item The formulas for period and fundamental difference for $w=15, 39, 45$ follow the same rule as the formulas for $w>24$, $w = 0 \mod 6$ proved by Finch. Like above, the initial evolution of $\Z_{(1,1,1)}(1,2,15)$, $\Z_{(1,1,1)}(1,2,39)$ and $\Z_{(1,1,1)}(1,2,45)$ is however different from the evolution of $\Z_{(1,1,1)}(1,2,w)$ for $w>24$, $w = 0 \mod 6$.
\item The sequences of differences for the cases $w=15, 39, 45$ and $w>24$, $w = 0 \mod 6$ consist of $1$'s and one $6w+1$. The sequence of differences for $w=9, 12, 18, 24$ consists of several $6w+1$'s and multiple $1$'s. 
\item All these sequences except for $w=18, 24$ become regular extremely fast (the transient phase takes at most several multiples of $w$). In the case of $w=18$, the transient phase takes $\approx 250$ terms, and in the case of $w=24$, it takes $\approx1400$ terms, which is still fast compared to regular Ulam sequences in which transient phase may even take $10^8$ terms, according to computations made by Finch \cite{finch_1992_1}. 
\end{enumerate}
\end{observation}

Based on the numerical analysis of $(1,1,1)$-sequences of the form $\ZZ_{(1,1,1)}(1,2,w)$, $\ZZ_{(1,1,1)}(3,4,w)$, and $\ZZ_{(1,1,1)}(5,6,w)$, there seem to be three categories of $(1,1,1)$-sequences:
\begin{enumerate}
\item regular $(1,1,1)$-sequences
\item $(1,1,1)$-sequences with quasiperiodic behavior and possibly multiple quasi-periods. We call a sequence $(a_n)$ $\textit{quasiperiodic}$ with $\textit{quasiperiod}$ $\lambda$ if the distribution of $(a_n)$ mod $\lambda$ converges to a limit distribution that is non-uniform and non-discrete. Steinerberger discovered that $\U(1,2)$ and many other Ulam sequences are quasiperiodic \cite{steinerberger_2016} with quasiperiod $\lambda\sim 2.4...$, as described in Section \ref{Background}.
\item sequences $(a_n)$ with the following properties: there exist natural numbers $n_1 < n_2 < n_3 <...$ s.t. $a_{n_2}-a_{n_1}\approx a_{n_3}-a_{n_2}\approx a_{n_4}-a_{n-3}\approx ...$ and $a_{n_{i+1}}-a_{n_{i+1} - 1}$ are much greater than all of $a_{n_i+1}-a_{n_i}$, $a_{n_i+2}-a_{n_i +1}$, ... $a_{n_{i+1}-1}-a_{n_{i+1}-2}$. We call such sequences quasi-regular because they behave "almost regularly".
\end{enumerate}

	\begin{figure}[h!]
	\includegraphics[height = 0.25\textheight]{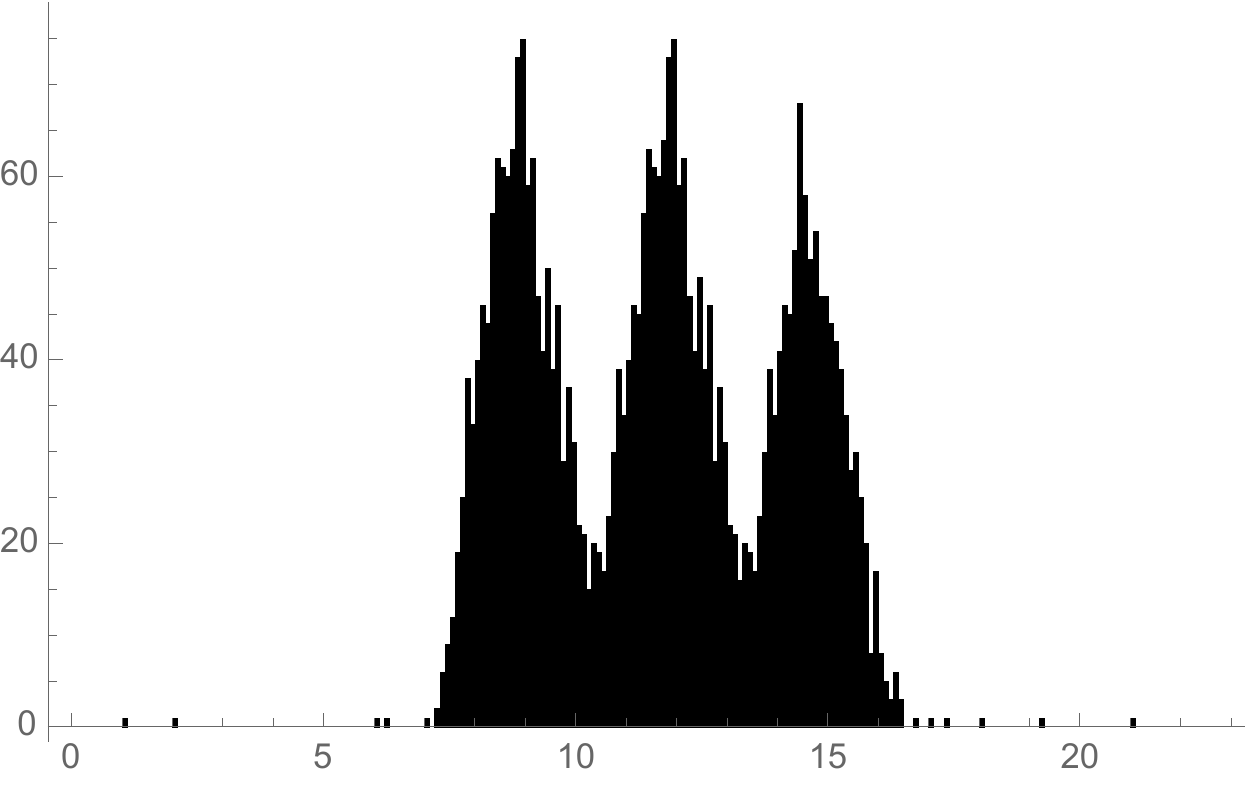}
	
	\caption{The distribution of $\Z_{(1,1,1)}(1,2,6) \mod 22.893$.}
	\label{U126}
	\end{figure}

Quasi-regularity seems to be a widespread phenomenon in $(1,1,1)$-sequences. For instance, most of $\Z_{(1,1,1)}(3,4,w)$ and $\Z_{(1,1,1)}(5,6,w)$ behave this way. 
We have made several observations concerning quasi-regular sequences:

\begin{observation}
\
\begin{enumerate}
\item Quasi-regular sequences do not exhibit quasiperiodic behavior. 
\item For fixed $u,v$, and $w$ in a fixed modulo class mod $u+v$, $\Z_{(1,1,1)}(u,v,w)=(a_n)$ has the following property: the average of $a_{n_{i+1}}-a_{n_i}$ is linear in $w$, where $(n_i)_i$ is a different sequence for each $w$ constructed as in Observation \ref{various (1,1,1)-sequences}. For example:
\begin{itemize}
\item for $\Z_{(1,1,1)}(3,4,w)$, $w=0 \mod 7$, $a_{n_{i+1}}-a_{n_i}\approx 31.145+10.559 w$
\item for $\Z_{(1,1,1)}(3,4,w)$, $w=1 \mod 7$, $a_{n_{i+1}}-a_{n_i}\approx 11.758 + 11.065 w$
\item for $\Z_{(1,1,1)}(5,6,w)$, $w=0 \mod 11$, $a_{n_{i+1}}-a_{n_i}\approx 39.111 + 15.088 w$
\item for $\Z_{(1,1,1)}(5,6,w)$, $w=2 \mod 11$, $a_{n_{i+1}}-a_{n_i}\approx 29.980 + 12.862 w$
\item for $\Z_{(1,1,1)}(6,11,w)$, $w=0 \mod 17$, $a_{n_{i+1}}-a_{n_i}\approx 90.979 + 18.822 w$
\end{itemize}
\end{enumerate}
\end{observation}

It may be the case that quasi-regular sequences are in fact regular with a long transient phase. An argument in favor of this assertion is that regular sequences originally behave quasi-regularly. Since we have only been able to generate at most several thousand terms of $(1,1,1)$-sequences, we have been unable to verify this assertion. 

\subsection{Questions}
The following list of questions provides possible directions for future research on $(1,1,1)$-sequences: 

\begin{enumerate}
\item Are quasi-regular sequences regular?
\item Is there a simple criterion that would allow us to determine which $(1,1,1)$-sequences are regular, similar to Finch's criterion for regularity of Ulam sequences (this question has already been asked by Finch in \cite{finch_1992_1})?
\item Can one make a general conjecture about which families of $(1,1,1)$-sequences are regular, quasi-regular, or quasiperiodic?
\item Do various quasi-periods of quasiperiodic $(1,1,1)$-sequences follow any recognizable patterns? Are quasi-periods of a fixed $(1,1,1)$-sequence independent or related? 

\end{enumerate}

\section{$(2,1)$-sequences}

In the previous section, we have observed that a lot of $(1,1,1)$-sequences are quasi-regular - a phenomenon that does not appear in Ulam sequences. $(2,1)$-sequences also have interesting properties, some of which are similar to the properties of Ulam sequences while others are distinct. 

%Of all the classes of sequences investigated in this paper, the least has been established about this one, even conjecturally. 

\subsection{$a,b$ odd}
\

First, we make two crucial observations. If $a,b$ are relatively prime and odd, then $\Z_{(2,1)}(a,b)$ will have no even terms. Moreover, if $z=2x+y = i \mod 4$ for odd $i$, then $y = i+2 \mod 4$. 

Such sequences seem to fit in one of the two categories:
\begin{enumerate}
\item $\Z_{(2,1)}(a,b)$ has finitely many terms in one of the odd residue classes mod 4. Moreover, all but finitely many positive integers in this residue class have multiple representations as a sum $2x+y$ for distinct $x,y\in \Z_{(2,1)}(a,b)$.
\item $\Z_{(2,1)}(a,b)$ has infinitely many terms in both odd residue classes mod 4, but one of them contains a larger proportion of terms of the sequence.
\end{enumerate}
\
The first category seems more prevalent, and its examples include:
\begin{itemize}
\item all $\Z_{(2,1)}(3,b)$ for $b$ odd and not divisible by 3 up to $b<45$
\item all $\Z_{(2,1)}(1,b)$ for $b$ odd, $b<50$ excluding $b=29,37,41,45$
\item $\Z_{(2,1)}(5,21)$, $\Z_{(2,1)}(5,23)$
\item $\Z_{(2,1)}(7,53)$
\end{itemize}

The only examples of the second catergory that we have found so far are $\Z_{(2,1)}(1,b)$ for $b=29,37,41,45,53,65,81$.
In first category, the category of $(2,1)$-sequences with finitely many terms in one of the odd residue classes, we encounter several interesting phenomena that we will now analyze.

\begin{example}
$\Z_{(2,1)}(1,3)$
\end{example}
First, one of these sequences is regular:

\begin{theorem}
$\Z_{(2,1)}(1,3)$ is regular. More precisely, 
$$\Z_{(2,1)}(1,3)=\{3,15\}\cup(4\ZZ_{\geqslant 0}+1)\setminus{\{9\}}=\{1, 3, 5, 13, 15, 17, 21, 25,...\}$$
\end{theorem}
\begin{proof}
First, note that the first 13 terms of the sequence are 1, 3, 5, 13, 15, 17, 21, 25, 29, 33, 37, 41, 45. I claim that for all $a\geqslant 45$, $a\in \Z_{(2,1)}(1,3) \iff a=1\; \mod\; 4$. Suppose $a = 1 \mod 4$, and the statement is true for all $45\leqslant x \leqslant a$. Both $a+2$ and $a+6$ have at least two representations:
\begin{align*}
a+2 &= 2\cdot 1 + a = 2\cdot 3 + (a-4)\\
a+6 &= 2\cdot 3 + a = 2\cdot 5 + (a-4)
\end{align*}
hence they are not in the sequence. It remains to show that $a+4$ is in the sequence. $a+4$ can have at most two representations: $a+4 = 2\cdot \frac{a+1}{2} + 3$ or $a+4 = 2\cdot \frac{a-11}{2} + 15$ because 3 and 15 are the only numbers equal $3\; \mod\; 4$ in $\Z_{(2,1)}(1,3)$. If $a = 1\; \mod\; 8$, then $\frac{a+1}{2} = 1\;\mod\; 4$ and $\frac{a-11}{2} = 3\;\mod\; 4$, hence $\frac{a+1}{2}$ is in the sequence but $\frac{a-11}{2}$ is not. If $a = 5\; \mod\; 8$, then the opposite is true. Since $a\geqslant 45$, we have $\frac{a+1}{2}>\frac{a-11}{2}\geqslant 17$, and by assumption we know that an odd number $x\geqslant 17$ is in the sequence iff $x = 1\;\mod 4$. Thus always precisely one of $\frac{a+1}{2},\frac{a-11}{2}$ is in the sequence, meaning that $a+4$ has precisely one representation, and so is in the sequence. 
\end{proof}

\begin{example}
$\Z_{(2,1)}(1,9)$ and $\Z_{(2,1)}(3,7)$
\end{example}

While $\Z_{(2,1)}(1,3)$ is the only $(2,1)$-sequence that has been found regular, two other sequences, $\Z_{(2,1)}(1,9)$ and $\Z_{(2,1)}(3,7)$, exhibit an even more unexpected behavior. For each of these sequences, there exist positive integers $n_0$ and $d$, and a sequence of integers $m_1<m_2<...$ s.t.:
\begin{itemize}
\item $m_k\approx n_0+2^{k-1}d$ (with a very small error)
\item $a_{m_k}-a_{m_k - 1}=2\cdot(a_{m_{k-1}}-a_{m_{k-1}-1})-12$, or equivalently $$a_{m_k}-a_{m_k - 1}=2^{k-1}(a_{m_1}-a_{m_1 - 1})-(2^{k-1}-1)\cdot 12$$
\item up to small irregularities, the sequence of differences 
$$(a_{m_{k-1}}-a_{m_{k-1}-1}, a_{m_{k-1}+1}-a_{m_{k-1}}, a_{m_{k-1}+2}-a_{m_{k-1}+1},..., a_{m_k}-a_{m_k - 1})$$ can be broken into chunks of $(4,4,8,12)$.
\end{itemize}

The tricky part, which we have not attempted to do, is to check whether these irregularities follow some nice patterns or behave erratically. If the former is true, one could attempt to write a closed form formula for the terms of these two sequences.

Assuming that the aforementioned observations hold for the entirety of the sequence, and in particular assuming that the aforementioned inequalities are negligible and the sequence of differences can be completely broken into chunks of $(4,4,8,12)$, we can make inferences about the asymptotic densities of these sequences. Note:

\begin{align*}
a_{m_k}&\approx a_{m_1}+\frac{m_k-m_1}{4}(4+4+8+12)+(a_{m_1}-a_{m_1-1})(2+4+8+...+2^{k-1})\\
&-12(2+4+8+...+2^{k-1}-(k-1))\\
&\approx a_{m_1}+\frac{2^{k-1}d}{4}\cdot 28 +(a_{m_1}-a_{m_1-1})\cdot 2\cdot(2^{k-2}-1)-12\cdot 2^{k-1}\\
&\approx 7\cdot 2^{k-1}d+2^{k-1}(a_{m_1}-a_{m_1-1})-12\cdot 2^{k-1}\\
&= 2^{k-1}(7d-12+a_{m_1}-a_{m_1-1})
\end{align*}

Thus:
\begin{align*}
\frac{m_k}{a_{m_k}}&\approx\frac{n_0+2^{k-1}d}{2^{k-1}(7d-12+a_{m_1}-a_{m_1-1})}\\
&\approx \frac{d}{7d-12+a_{m_1}-a_{m_1-1}}
\end{align*}

For $\Z_{(2,1)}(1,9)$, we have $d\approx 6.86$, therefore $a_{m_1}-a_{m_1-1}=28$, hence $m_k/a_{m_k}\approx 6.86/64.02\approx 0.107$.

For $\Z_{(2,1)}(3,7)$, $d\approx 4.85$ while $a_{m_1}-a_{m_1-1}=24$, hence $m_k/a_{m_k}\approx 4.85/45.95\approx 0.106$. 

For $0<n<m_{k+1}-m_k$, 
$$\frac{m_k+n}{a_{m_k+n}}\approx\frac{2^{k-1}d+n}{2^{k-1}(7d-12+a_{m_1}-a_{m_1-1})+7n}$$
where 7 is the average of 4,4,8, and 12. Note that $(m_k+n)/a_{m_k+n}$ is increasing in $n$ because $0.107, 0.106<1/7$. Thus, the maximum of $(m_k+n)/a_{m_k+n}$ is obtained approximately when $n=m_{k+1}-m_k-1=2^{k-1}d-1\approx 2^{k-1}d$, in which case 

\begin{align*}
\frac{m_k+n}{a_{m_k+n}}&\approx\frac{2^{k-1}d+n}{2^{k-1}(7d-12+a_{m_1}-a_{m_1-1})+7n}\\
&\approx\frac{2^k d}{2^{k-1}(7d-12+a_{m_1}-a_{m_1-1}+7d)}\\
&=\frac{2d}{14d-12+a_{m_1}-a_{m_1-1}}
\end{align*}

For $\Z_{(2,1)}(1,9)$, this fraction equals ${2\cdot6.86}/{112.04}\approx 0.123$. For $\Z_{(2,1)}(3,7)$, it equals ${2\cdot4.85}/{79.90}\approx 0.121$. Hence we can say the following: $\Z_{(2,1)}(1,9)$ has lower density of roughly $0.107$ and upper density of roughly $0.122$, and $\Z_{(2,1)}(3,7)$ has lower density of roughly $0.106$ and upper density of roughly $0.121$. This result, contingent on the veracity of the aforementioned observations, is consistent with the numerical data shown on Figure \ref{density19}.

\begin{example}
$\Z_{(2,1)}(1,b)$ for $b = 3\mod 16$
\end{example}

Hinman, Kuca, Schlesinger, and Sheydvasser's rigidity conjecture states that Ulam sequences $\U(a,b)$ follow the same pattern whenever $b$ is in a fixed modulo class of some multiple $L$ of $a$ and $b>b_0$ for some $b_0\in\NN$ - a result well grounded in numerical data \cite{hinman_kuca_schlesinger_sheydvasser_2017}. We have looked at $(2,1)$-sequences to find similar patterns. However, we have only found one clear pattern: If $a=1$ and $b = 3 \mod 16$ for $b\geqslant 35$, then:
\begin{itemize}
\item $\Z_{(2,1)}(a,b)\;\cap\; (3+4\ZZ_{\geqslant 0})=\{b, b+4, b+8,...,2b-3; 2b+5, 2b+13, 2b+21, ..., 3b-6; 9b-36, 9b-16, 11b-42\}$
\item The only numbers in $3+4\ZZ_{\geqslant 0}$ with no representation are $3,7,11,..,b-4; 9b-12, 9b-8, 9b-4, ..., 11b-46$
\end{itemize}

In Ulam and modified Ulam sequences, a finite number of even terms implied regularity; do the restrictions in odd residue classes mod 4 in $(2,1)$-sequences also induce some kind of "regularity"? We have seen one example of a regular $(2,1)$-sequence, and 2 examples of $(2,1)$-sequences whose behavior also seems very "regular" (parentheses are used to distinguish our definition of regularity from an intuitive use of the word). Is there another notion of "regularity" that sequences with all but finitely many terms in an odd residue class mod 4 would satisfy?

\subsection{one of $a,b$ is even}
\

If one of $a,b$ is even and the other is odd, we also end up with two classes of sequences:
\begin{itemize}
\item quasiperiodic sequences with positive density: e.g. $\Z_{(2,1)}(1,10)$, $\Z_{(2,1)}(2,7)$ 
\item sequences with no apparent quasiperiodic behavior, density converging to 0, and usually no big disproportions in how often each residue class mod 4 is represented: e.g. $\Z_{(2,1)}(1,2)$, $\Z_{(2,1)}(1,24)$, $\Z_{(2,1)}(4,11)$, $\Z_{(2,1)}(7,12)$.
\end{itemize}

The surprising phenomenon of quasiperiodicity that appears in Ulam sequences is thus present as well in $(2,1)$-sequences; perhaps more unexpectedly, there seems to be a fairly large collection of $(2,1)$-sequences that exhibit no properties of interest to us - neither quasiperiodicity nor positive density. 

	\begin{figure}[h!]
	\includegraphics[height = 0.25\textheight]{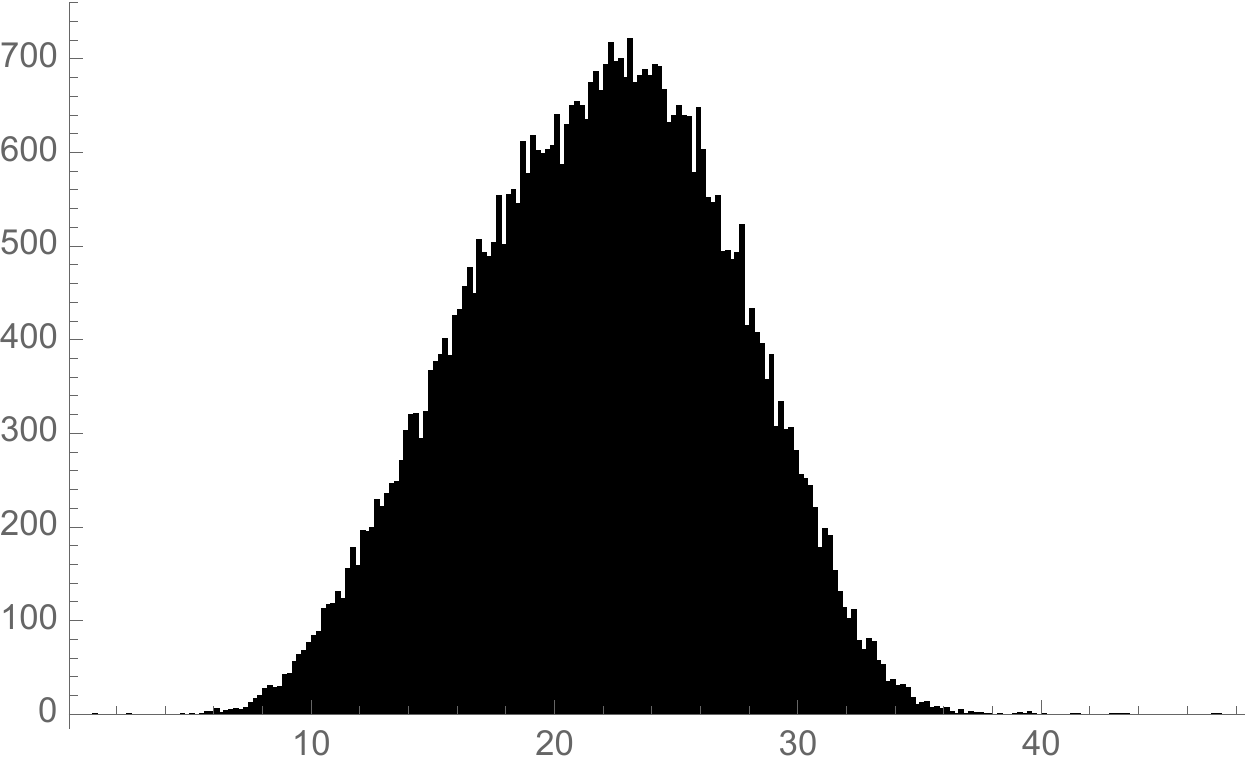}
	
	\caption{The distribution of $\Z_{(2,1)}(1,10) \mod 96.605$.}
	\label{U110}
	\end{figure}

\subsection{Questions}
\ 

There is a number of open questions about $(2,1)$-sequences:
\begin{enumerate}
\item Are $(2,1)$-sequences other than $\Z_{(2,1)}(1,3)$ regular? 
\item Do there exist $(2,1)$-sequences other than $\Z_{(2,1)}(1,9)$ and $\Z_{(2,1)}(3,7)$ that have similar behavior to these sequences? More generally, given any $n>1$, will there exist an $(n,1)$-sequence with the following property: there exist positive integers $n_0, c, d$, and a sequence of integers $m_1<m_2<...$ s.t.:
\begin{itemize}
\item $m_k\approx n_0+n^{k-1}d$ (with a very small error)
\item $a_{m_k}-a_{m_k - 1}=n\cdot(a_{m_{k-1}}-a_{m_{k-1}-1})+c$
\item up to small irregularities, the sequence of differences 
$$(a_{m_{k-1}}-a_{m_{k-1}-1}, a_{m_{k-1}+1}-a_{m_{k-1}}, a_{m_{k-1}+2}-a_{m_{k-1}+1},..., a_{m_k}-a_{m_k - 1})$$ can be broken into small chunks of $(b_1,...,b_l)$ for some positive integers $b_1,...,b_l\in\ZZ_+$.
\end{itemize}
\item Is there a different notion of "regularity" that sequences with all but finitely many terms in an odd residue class will satisfy?
\item Is there a simple classification of which $(2,1)$-sequences behave in any of the ways described in this section?
\end{enumerate}

\subsection*{Acknowledgements}
This paper would not be possible without the SUMRY 2017 research project on Ulam sequences that the author did with Joshua Hinman, Alexander Schlesinger, and Arseniy Sheydvasser. Moreover, the author is indebted to Stefan Steinerberger for mentorship along the way and to Ross Berkowitz and Patrick Devlin for useful suggestions.

\bibliography{Ulam_Library}
\bibliographystyle{alpha}

\end{document}